\documentclass[12pt]{amsart}

\usepackage{amssymb}

\newtheorem{theorem}{Theorem}[section]
\newtheorem{proposition}[theorem]{Proposition}
\newtheorem{lemma}[theorem]{Lemma}
\newtheorem{corollary}[theorem]{Corollary}

\theoremstyle{definition}
\newtheorem{definition}[theorem]{Definition}
\newtheorem{assumption}[theorem]{Assumption}
\newtheorem{remark}[theorem]{Remark}

\numberwithin{equation}{section}

\begin{document}
\baselineskip=15pt

\title[Pseudo-real Higgs $G$-bundles on K\"ahler manifolds]{Pseudo-real
principal Higgs bundles on compact K\"ahler manifolds}

\author[I. Biswas]{Indranil Biswas}

\address{School of Mathematics, Tata Institute of Fundamental
Research, Homi Bhabha Road, Bombay 400005, India}

\email{indranil@math.tifr.res.in}

\author[O. Garc\'{\i}a-Prada]{Oscar Garc\'{\i}a-Prada}

\address{Instituto de Ciencias Matem\'aticas, C/ Nicol\'as
Cabrera, no. 13--15, Campus Cantoblanco, UAM, 28049 Madrid,
Spain}

\email{oscar.garcia-prada@icmat.es}

\author[J. Hurtubise]{Jacques Hurtubise}

\address{Department of Mathematics, McGill University, Burnside
Hall, 805 Sherbrooke St. W., Montreal, Que. H3A 2K6, Canada}

\email{jacques.hurtubise@mcgill.ca}



\date{}

\begin{abstract}
Let $X$ be a compact connected K\"ahler manifold equipped with an anti-holomorphic
involution which is compatible with the K\"ahler structure. Let $G$ be a
connected complex reductive affine algebraic group equipped with a real
form $\sigma_G$. We define pseudo-real
principal $G$--bundles on $X$; these are generalizations of real algebraic
principal $G$--bundles over a real algebraic variety. Next we define stable,
semistable and polystable pseudo-real principal $G$--bundles. Their
relationships with the usual stable, semistable and polystable principal
$G$--bundles are investigated. We then prove that the following
Donaldson--Uhlenbeck--Yau type correspondence holds: a pseudo-real
principal $G$--bundle admits a compatible Einstein-Hermitian connection if and
only if it is polystable. A bijection between the following two sets
is established:
\begin{enumerate}
\item The isomorphism classes of polystable pseudo-real principal $G$--bundles
such that all the rational characteristic classes of the underlying
topological principal $G$--bundle vanish.

\item The equivalence classes of twisted representations of the extended
fundamental group of $X$ in a $\sigma_G$--invariant maximal compact subgroup of
$G$. (The twisted representations are defined using the central element in
the definition of a pseudo-real principal $G$--bundle.)
\end{enumerate}

All these results are also generalized to the pseudo-real Higgs $G$--bundle.
\medskip

\noindent
\textsc{R\'esum\'e.}\,
Soit $X$ une vari\'et\'e k\"ahlerienne compacte et connexe, \'equip\'ee d'une 
involution antiholomorphe compatible avec la structure K\"ahlerienne. Soit $G$ un 
groupe alg\'ebrique affine complexe, connexe et munie d'une forme r\'eelle $\sigma_G$. 
Nous d\'efinissons des $G$--fibr\'es principaux holomorphes pseudo-r\'eels sur $X$, ce 
qui g\'en\'eralise la notion de $G$--fibr\'e principal r\'eel sur une vari\'et\'e 
r\'eelle. Nous introduisons ensuite les notions de $G$-fibr\'e principal pseudo-r\'eel 
stable, semi-stable et polystable. La relation de ces concepts avec les notion 
usuelles de $G$-fibr\'e principal stable, semi-stable et polystable est discut\'ee. 
Nous d\'emontrons ensuite qu'il existe une correspondance de type 
Donaldson--Uhlenbeck--Yau: un $G $-fibr\'e principal holomorphe pseudo-r\'eel admet 
une connection Hermite-Einstein compatible si et seulement si il est polystable. Nous 
\'etablissons ensuite une bijection entre les deux ensembles suivants:
\begin{enumerate}
\item Les classes d'isomorphisme de $G$--fibr\'es principaux holomorphes 
pseudo-r\'eels sur $X$, dont toutes les classes caract\'eristiques rationnelles du 
$G$-fibr\'e topologique sous-jacent s'annulent;

\item Les classes d'\'equivalence de repr\'esentations tordues du groupe fondamental 
\'etendu de $X$ dans un sous-groupe maximal compact $\sigma_G$--invariant de $G$. (Les 
repr\'esentations tordues sont d\'efinies en utilisant l'\'el\'ement central qui entre 
dans la d\'efinition d'un $G$-fibr\'e principal pseudo-r\'eel.)
\end{enumerate}

Tous ces r\'esultats sont ensuite g\'en\'eralis\'es au cas du $G$--fibr\'e de Higgs 
pseudo-r\'eel.
\end{abstract}

\maketitle

\section{Introduction}\label{sec1}

Let $G$ be a connected reductive affine algebraic group defined over $\mathbb C$.
Let $$\sigma_G\, :\, G\, \longrightarrow\, G$$ be 
a real form on $G$. Fix a maximal compact subgroup $K_G\, \subset\, G$
such that $\sigma_G(K_G)\,=\, K_G$. Also, fix an element $c$ in the center of $K_G$
such that $\sigma_G(c)\,=\,c$. Let $(X\, ,\omega)$ be a compact connected
K\"ahler manifold equipped with an anti-holomorphic involution $\sigma_X$
such that $\sigma^*_X\omega\,=\, -\omega$.

Using $c$, we define pseudo-real principal $G$--bundles on $X$ (see
Definition \ref{def1}). We define stable,
semistable and polystable pseudo-real principal $G$--bundles on $X$.
These are related to the usual semistable and polystable principal $G$--bundles 
in the following way:

\begin{proposition}\label{prop-i-1}
A pseudo-real principal $G$--bundle $(E_G\, ,\rho)$ on $X$ is semistable
(respectively, polystable) if and only if the underlying holomorphic
principal $G$--bundle $E_G$ is semistable (respectively, polystable).
\end{proposition}

Proposition \ref{prop-i-1} is proved in Lemma \ref{lem1}, Lemma \ref{lem2}
and Corollary \ref{cor3}.

\begin{theorem}\label{th-i-1}
Let $(E_G\, ,\rho)$ be a pseudo-real principal $G$--bundle on $X$. The
following two statements are equivalent:
\begin{enumerate}
\item $(E_G\, ,\rho)$ is polystable.

\item The holomorphic principal $G$--bundle $E_G$ has an Einstein--Hermitian
reduction of structure group $E_{K_G}\, \subset\, E_G$ to the maximal
compact subgroup $K_G$ such that $\rho(E_{K_G})\,=\, E_{K_G}$.
\end{enumerate}
\end{theorem}

Theorem \ref{th-i-1} is proved in Corollary \ref{cor1} and Proposition
\ref{prop3}.

Fix a point $x_0\, \in\, X$ such that $\sigma_X(x_0)\, \not=\, x_0$.
Let $\Gamma(X,x_0)$ be the homotopy classes of paths
originating from $x_0$ that end in
either $x_0$ or $\sigma_X(x_0)$. It is a group that fits in a short exact
sequence
$$
e\,\longrightarrow\, \pi_1(X,x_0)\,\longrightarrow\,\Gamma(X,x_0)\,
\stackrel{\eta}{\longrightarrow}\,
{\mathbb Z}/2{\mathbb Z}\,\longrightarrow\, e\, .
$$
Let $\widetilde{K}\, =\, K_G\rtimes ({\mathbb Z}/2{\mathbb Z})$ be the
semi-direct product constructed using the involution $\sigma_G$ of $K_G$.
Let $\text{Map}'(\Gamma(X,x_0)\, , {\widetilde K})$ be the space of all maps
$\delta\, :\, \Gamma(X,x_0)\,\longrightarrow\, {\widetilde K}$
such that $\delta^{-1}(K_G)\,=\, \pi_1(X,x_0)$. We will write
${\mathbb Z}/2{\mathbb Z}\,=\, \{0\, ,1\}$. Let
$\text{Hom}_c(\Gamma(X,x_0)\, , {\widetilde K})$ be the space of all maps
$\delta\, \in\, \text{Map}'(\Gamma(X,x_0)\, , {\widetilde K})$
such that
\begin{itemize}
\item the restriction of $\delta$ to $\pi_1(X,x_0)$ is a homomorphism of groups,

\item $\delta(g'g)\,=\, c\delta(g')\delta(g)$, if $\eta(g)\, =\,1\,=\,
\eta(g')$, where $\eta$ is the above homomorphism, and

\item $\delta(g'g)\,=\, \delta(g')\delta(g)$ 
if $\eta(g)\cdot \eta(g')\, =\,0$.
\end{itemize}
Two elements $\delta'\, ,\delta'\, \in\, \text{Hom}_c(\Gamma(X,x_0)\, , {\widetilde
K})$ are called equivalent if there is an element $g\, \in\, K_G$ such
that $\delta'(z)\,=\, g^{-1}\delta(z)g$ for all $z\, \in\, \Gamma(X,x_0)$.

We prove the following (see Theorem \ref{thm3}):

\begin{theorem}\label{th-i-2}
There is a natural bijective correspondence between the equivalence classes of
elements of ${\rm Hom}_c(\Gamma(X,x_0)\, , {\widetilde K})$, and the isomorphism
classes of polystable pseudo-real principal $G$--bundles $(E_G\, ,\rho)$
satisfying the following two conditions:
\begin{itemize}
\item $\int_X c_2({\rm ad}(E_G))\wedge\omega^{\dim_{\mathbb C}(X) -2}\,=\, 0$, and

\item for any character $\chi$ of $G$, the line bundle over $X$ associated to
$E_G$ for $\chi$ is of degree zero.
\end{itemize}
\end{theorem}

It may be mentioned that a polystable principal $G$--bundle $E_G$ satisfies the
above two numerical conditions if and only if all the rational characteristic
classes of $E_G$ of positive degree vanish.

In Section \ref{sec-Higgs}, we extends the above results to the more general
context of pseudo-real principal $G$--bundle on $X$ equipped with a Higgs field
compatible with the pseudo-real structure.
We prove the following (see Proposition \ref{cor-h1} and
Proposition \ref{cor-h2}):

\begin{proposition}\label{i-h1}
Let $(E_G\, ,\rho\, ,\theta)$ be a pseudo-real principal Higgs $G$--bundle. Then
the principal Higgs $G$--bundle $(E_G\, ,\theta)$ admits an Einstein--Hermitian
structure $E_{K_G}\, \subset\, E_G$ with $\rho(E_{K_G})\,=\, E_{K_G}$
if and only if $(E_G\, ,\rho\, ,\theta)$ is polystable.
\end{proposition}

The definition of an Einstein--Hermitian structure on a
principal Higgs $G$--bundle is recalled in Definition \ref{def-h-eh}.

Let ${\widetilde G} \,:=\, G\rtimes ({\mathbb Z}/2{\mathbb Z})$ be the semi-direct
product constructed using $\sigma_G$. Define ${\rm Hom}_c(\Gamma(x_0)\, ,
{\widetilde G})$ as before by replacing ${\widetilde K}$ with $\widetilde G$. See
Section \ref{sec-Higgs} for the equivalence classes of
completely reducible elements of ${\rm Hom}_c(\Gamma(x_0)\, , {\widetilde G})$.

\begin{proposition}\label{i-h2}
There is a natural bijective correspondence between the equivalence classes of
completely reducible elements of ${\rm Hom}_c(\Gamma(x_0)\, , {\widetilde G})$,
and the isomorphism classes of polystable pseudo-real principal Higgs $G$--bundles
$(E_G\, ,\rho\, ,\theta)$ satisfying the following conditions:
\begin{itemize}
\item $\int_X c_2({\rm ad}(E_G))\wedge\omega^{\dim_{\mathbb C}(X) -2}\,=\,
0$, and

\item for any character $\chi$ of $G$, the line bundle over $X$ associated to
$E_G$ for $\chi$ is of degree zero.
\end{itemize}
\end{proposition}

When $X$ is a compact Riemann surface, some of the above results were obtained
in \cite{BH}.

A comment on the definition of (semi)stability is in order.
As explained in \cite[Section 2.3]{BH}, when the base field is $\mathbb R$
(more generally, when it is not algebraically closed),
the definition in \cite{Be}, and not the one in \cite{Ra}, is the right one.
Therefore, we have to follow the definition of \cite{Be} here.

\section{Pseudo-real principal bundles}\label{sec2}

Let $X$ be a compact connected K\"ahler manifold. The real tangent bundle of $X$
will be denoted by $T^{\mathbb R}X$. The almost complex structure on $X$, which
is a smooth section of $End(T^{\mathbb R}X)\,=\,(T^{\mathbb R}X)\otimes
(T^{\mathbb R}X)^*$, will be denoted by $J$. Let
$$
\sigma_X\, :\, X\, \longrightarrow\, X
$$
be a diffeomorphism such that
$d\sigma_X\circ J\,=\, - J\circ d\sigma_X$, where
\begin{equation}\label{e-1}
d\sigma_X\, :\,T^{\mathbb R}X\, \longrightarrow\, \sigma^*_X T^{\mathbb R}X
\end{equation}
is the differential of $\sigma_X$.

Let $\omega$ be a K\"ahler form on $X$. The inner product
on $T^{\mathbb R}X$ corresponding to $\omega$ will be denoted by
$\widetilde\omega$. The K\"ahler form $\omega$ is said to be
\textit{compatible} with $\sigma_X$ if $d\sigma_X$ preserves
$\widetilde\omega$. It is straightforward to check that $\omega$ is compatible
with $\sigma_X$ if and only if $\sigma^*_X\omega\,=\, -\omega$.

The K\"ahler manifold $X$ admits a K\"ahler form compatible with
$\sigma_X$. To see this, take any K\"ahler form $\omega$ on $X$, and define
$\widetilde\omega$ as above. Let $\widetilde{\widetilde\omega}$ be the Riemannian
metric on $X$ defined by 
$$
\widetilde{\widetilde\omega}(v\, ,w)\,:=\, {\widetilde\omega}(v\, ,w)+
{\widetilde\omega}(d\sigma_X(v)\, ,d\sigma_X(w))
$$
($d\sigma_X$ is defined in \eqref{e-1}).
Since $d\sigma_X\circ J\,=\, - J\circ d\sigma_X$, and $J$ is orthogonal
with respect to $\widetilde\omega$, it follows that $\widetilde{\widetilde\omega}$
also defines a K\"ahler structure on $X$. In fact, the K\"ahler form for
$\widetilde{\widetilde\omega}$ is $\omega-\sigma^*_X \omega$, hence
the K\"ahler form is closed. This K\"ahler structure defined by
$\widetilde{\widetilde\omega}$ is clearly compatible with $\sigma_X$.

Fix a K\"ahler form $\omega$ on $X$ compatible with $\sigma_X$.
For a torsionfree coherent analytic sheaf
$F$ on $X$, define
\begin{equation}\label{dde}
\text{degree}(F)\, :=\, \int_X c_1(F)\omega^{\dim_{\mathbb C}(X)-1}
\, \in\, {\mathbb R}\, .
\end{equation}

Let $G$ be a connected reductive affine algebraic group
defined over $\mathbb C$. We fix a real form $\sigma_G$ of $G$. This means
that
$$
\sigma_G\, :\, G\, \longrightarrow\, G
$$
is an anti-holomorphic isomorphism of order two.
The Lie algebra of $G$ will be denoted by $\mathfrak g$.
The center of $G$ will be denoted by $Z_G$. Let
$$
Z_{\mathbb R}\, :=\, Z_G \cap G^{\sigma_G}
$$
be the group of fixed points in $Z_G$ for the involution $\sigma_G$.

Let $E_G$ be a holomorphic principal $G$--bundle over $X$. By
$\overline{E}_G$ we denote the $C^\infty$ principal $G$--bundle over $X$
obtained by extending the structure group of $E_G$ using the homomorphism
$\sigma_G$:
$$
\overline{E}_G \,=\, E_G\times^{\sigma_G} G\, .
$$
In other words, $\overline{E}_G$ is the quotient of $E_G\times G$ where
two points $(z_1\, ,g_1)$ and $(z_1\, ,g_1)$ are identified if there is an
element $g\, \in\, G$ such that $z_2\, =\, z_1g$ and
$g_2\, =\, \sigma_G(g)^{-1}g_1$. The total space of
$\overline{E}_G$ is canonically identified with the total space of $E_G$;
this identification $\overline{E}_G\, \longrightarrow\, E_G$
sends the equivalence class of $(z\, ,g)$ to $z\sigma_G(g)$
(see \cite[Remark 2.1]{BH}). The pullback
$\sigma^*_X\overline{E}_G$ is a holomorphic principal $G$--bundle over
$X$, although $\overline{E}_G$ is not equipped with a holomorphic
structure. The holomorphic structure is uniquely determined by the following
condition: a section of $\sigma^*_X\overline{E}_G$ defined over an open subset
$U\, \subset\, X$ is holomorphic if and only if the corresponding
section of $E_G$ over $\sigma_X(U)$ is holomorphic.

\begin{definition}\label{def1}
A {\em pseudo-real} principal $G$--bundle on $X$ is a pair
of the form $(E_G\, ,\rho)$, where $E_G\, \longrightarrow\, X$ is a
holomorphic principal $G$--bundle, and
$$
\rho\, :\, E_G\, \longrightarrow\, \sigma^*_X\overline{E}_G
$$
is a holomorphic isomorphism of principal $G$--bundles satisfying
the condition that there is an element $c\, \in\, Z_{\mathbb R}$
such that the composition
$$
E_G\, \stackrel{\rho}{\longrightarrow}\, \sigma^*_X\overline{E}_G
\, \stackrel{\sigma^*_X\overline{\rho}}{\longrightarrow}\,
\sigma^*_X\overline{\sigma^*_X\overline{E}}_G\,=\,
\sigma^*_X\sigma^*_X\overline{\overline{E}}_G\, =\, E_G
$$
coincides with the automorphism of $E_G$ defined by $z\, \longmapsto\,
zc$.

If $(E_G\, ,\rho)$ is a pseudo-real principal $G$--bundle such that
$c\,=\,e$, then it is called a {\em real principal $G$--bundle}.
\end{definition}

Using the $C^\infty$ canonical identification between $E_G$ and $\overline{E}_G$, the
isomorphism $\rho$ in Definition \ref{def1} produces an anti-holomorphic diffeomorphism
of the total space of $E_G$ over the involution $\sigma_X$. 
This diffeomorphism of $E_G$ will also be
denoted by $\rho$. Clearly, we have
\begin{equation}\label{r-m}
\rho(zg)\, =\, \rho(z)\sigma_G(g)
\end{equation}
for all $z\, \in\, E_G$ and $g\, \in\, G$. Also, $\rho^2(z)\,=\, zc$, where $c$
is the element in Definition \ref{def1}.

An \textit{isomorphism} between
two pseudo-real principal $G$--bundles $(E_G\, ,\rho)$ and
$(F_G\, ,\delta)$ is a holomorphic isomorphism of principal
$G$--bundles
$$
\mu\, :\, E_G\, \longrightarrow\, F_G
$$
such that the following diagram commutes:
$$
\begin{matrix}
E_G & \stackrel{\rho}{\longrightarrow} & \sigma^*_X\overline{E}_G\\
~\Big\downarrow\mu && ~\,\text{  }~\,\text{ }\Big\downarrow \sigma^*_X\overline{\mu}\\
F_G & \stackrel{\delta}{\longrightarrow} & \sigma^*_X\overline{F}_G
\end{matrix}
$$
where $\sigma^*\overline{\mu}$ is the holomorphic isomorphism of
principal $G$--bundles given by $\mu$; the map $\sigma^*\overline{\mu}$
coincides with $\mu$ using the above mentioned identification of the total
spaces of $E_G$ and $F_G$ with those of $\sigma^*\overline{E}_G$
and $\sigma^*\overline{F}_G$ respectively.

Let
$$
{\rm Ad}(E_G)\,:=\, E_G\times^G G\, \longrightarrow\, X
$$
be the holomorphic
fiber bundle associated to $E_G$ for the adjoint action of $G$ on itself.
So ${\rm Ad}(E_G)$ is the quotient of $E_G\times G$ where two points
$(z_1\, ,g_1)\, ,(z_2\, ,g_2)$ are identified if there is an element $g\, \in\,
G$ such that $z_2\,=\, z_1g$ and $g_2\,=\, g^{-1}g_1g$. Therefore,
the fibers of ${\rm Ad}(E_G)$ are groups identified with $G$ up to inner
automorphisms. The fiber of ${\rm Ad}(E_G)$ over any point $x\, \in\, X$ is
identified with the space of all automorphisms of the fiber $(E_G)_x$ that
commute with the action of $G$ on $(E_G)_x$; this identification is
constructed as follows: the action of
$(z_1\, ,g_1)\, \in\, (E_G)_x\times G$ on $(E_G)_x$ is $z_1g\, \longmapsto\,
z_1g_1g$.

Let
$$
{\rm Ad}(\overline{E}_G)\,:=\, \overline{E}_G\times^G G\, \longrightarrow\, X
$$
be the $C^\infty$ fiber bundle associated to $\overline{E}_G$ for the adjoint
action of $G$ on itself. The homomorphism $\sigma_G$ produces a
$C^\infty$ isomorphism of fiber bundles
$$
\alpha_E\, :\, {\rm Ad}(E_G)\,\longrightarrow\, {\rm Ad}(\overline{E}_G)
$$
whose restriction to each fiber is an isomorphism of groups. More precisely,
$\alpha_E$ sends the equivalence class of $(z\, ,g)\, \in\, E_G\times G$
to the equivalence class of $(z\, ,\sigma_G(g))\, \in\, \overline{E}_G\times G$
(recall that the fibers of $E_G$ and $\overline{E}_G$ are naturally
identified). The isomorphism $\rho$ in Definition \ref{def1} produces an
isomorphism
$$
\rho''\, :\, \text{Ad}(E_G)\, \longrightarrow\, \text{Ad}(\sigma^*_X\overline{E}_G)
\,=\, \sigma^*_X{\rm Ad}(\overline{E}_G)
$$
which is holomorphic. Let
$$
(\sigma^*_X\alpha^{-1}_E)\circ \rho''\, :\, 
{\rm Ad}(E_G)\,\longrightarrow\,\sigma^*_X{\rm Ad}(E_G)
$$
be the composition. It defines a $C^\infty$--isomorphism of fiber bundles
\begin{equation}\label{rp}
\rho'\, :\, {\rm Ad}(E_G)\,\longrightarrow\,{\rm Ad}(E_G)
\end{equation}
over the map $\sigma_X$. This map $\rho'$ is an anti-holomorphic
involution, and it preserves the group--structure of the fibers
of ${\rm Ad}(E_G)$. That $\rho'$ is indeed an involution follows
immediately from the fact that
the adjoint action of $c\, \in\, {\mathbb Z}_{\mathbb R}$ (see Definition
\ref{def1}) on $G$ is trivial.

Let
$$
{\rm ad}(E_G)\,:=\, E_G\times^G {\mathfrak g}\, \longrightarrow\, X
$$
be the holomorphic vector bundle associated to $E_G$ for the adjoint action of
$G$ on $\mathfrak g$. It is the Lie algebra bundle corresponding to
${\rm Ad}(E_G)$. The anti-holomorphic involution $\rho'$ in \eqref{rp} produces
an anti-holomorphic automorphism of order two of the vector bundle ${\rm ad}(E_G)$
\begin{equation}\label{rp2}
\widetilde{\rho}\, :\, {\rm ad}(E_G)\, \longrightarrow\, {\rm ad}(E_G)
\end{equation}
over $\sigma_X$. To describe $\widetilde{\rho}$ explicitly, recall that
${\rm ad}(E_G)$ is the quotient of $E_G\times \mathfrak g$ where
two points $(z_1\, ,v_1)$ and $(z_2\, ,v_2)$ of
$E_G\times \mathfrak g$ are identified if there is an
element $g\, \in\, G$ such that $z_2\,=\, z_1g$ and $v_2\,=\, \text{Ad}(g)(v_1)$
(the automorphism $\text{Ad}(g)$ of $\mathfrak g$ is the differential
at identity of the
automorphism of $G$ defined by $g'\, \longmapsto\, g^{-1}g'g$). Let
\begin{equation}\label{dsg}
d\sigma_G\,:\, {\mathfrak g}\, \longrightarrow\, {\mathfrak g}
\end{equation}
be the differential at identity of $\sigma_G$. The
anti-holomorphic automorphism of $E_G\times \mathfrak g$ defined by $\rho\times
d\sigma_G$ descends to an anti-holomorphic automorphism of the quotient
${\rm ad}(E_G)$; this automorphism of ${\rm ad}(E_G)$ will be denoted by
$\widetilde{\rho}$. Since the adjoint action of $Z_G$ on $\mathfrak g$ is trivial,
it follows that $\widetilde{\rho}$ is of order two. This map $\widetilde{\rho}$
preserves the Lie algebra structure of the fibers of ${\rm ad}(E_G)$.
The homomorphism in \eqref{rp2} coincides with $\widetilde{\rho}$.

For a holomorphic vector bundle $V$ on $X$, by $\overline{V}$ we will denote the
$C^\infty$ vector bundle whose underlying real vector bundle is identified with
that of $V$, while multiplication by a complex number $\lambda$
on $\overline{V}$ coincides with the multiplication by
$\overline{\lambda}$ on $V$. If $E_{\rm GL}$ is the principal $\text{GL}(r,
{\mathbb C})$--bundle associated to $V$, where $r\,=\, \text{rank}(V)$,
then $\overline{V}$ corresponds to $\overline{E}_{\rm GL}$.
The pullback $\sigma^*_X \overline{V}$ has
a natural holomorphic structure; a section of $\sigma^*_X\overline{V}$ defined over
an open subset $U\, \subset\, X$ is holomorphic if and only if the corresponding
section of $V$ over $\sigma_X(U)$ is holomorphic.

Note that $\widetilde{\rho}$ in \eqref{rp2} coincides with the holomorphic
isomorphism
$$
{\rm ad}(E_G)\, \longrightarrow\, {\rm ad}(\sigma^*\overline{E}_G)\,=\,
\sigma^*_X\overline{{\rm ad}(E_G)}
$$
given by $\rho$ in Definition \ref{def1} after we use the above conjugate linear
identification of ${\rm ad}(E_G)$ with $\overline{{\rm ad}(E_G)}$ together with
the natural identification between the total spaces of $\sigma^*_X{\rm ad}(E_G)$
and ${\rm ad}(E_G)$.

A complex linear subspace
$S\, \subset\, {\rm ad}(E_G)_x$ is called a \textit{parabolic subalgebra} if $S$ is
the Lie algebra of a parabolic subgroup of ${\rm Ad}(E_G)_x$ (a connected
Zariski closed subgroup $P$ of ${\rm Ad}(E_G)_x$ is parabolic if
${\rm Ad}(E_G)_x/P$ is compact). A holomorphic
subbundle $F\,\subset\, {\rm ad}(E_G)\vert_U$ defined over an open subset
$U\, \subset\, X$ is
called a \textit{parabolic subalgebra bundle} if for each point $x\, \in\, U$, the
fiber $F_x$ is a parabolic subalgebra of ${\rm ad}(E_G)_x$.

\begin{definition}\label{def3}
A pseudo-real principal $G$--bundle $(E_G\, ,\rho)$
over $X$ is called {\em semistable} (respectively, {\em stable})
if for every pair of the form $(U\, ,{\mathfrak p})$, where
\begin{itemize}
\item $\iota_U\, :\, U\, \hookrightarrow\, X$ is a dense open subset
with $\sigma_X(U)\,=\, U$ such that the complement $X\setminus U$
is a closed complex analytic subset of $X$ of (complex)
codimension at least two, and

\item ${\mathfrak p}\, \subsetneq\, {\rm ad}(E_G)\vert_U$ is a parabolic subalgebra
bundle over $U$ such that $\widetilde{\rho}({\mathfrak p})\,=\,{\mathfrak p}$
(see \eqref{rp2} for $\widetilde{\rho}$), and the direct image $\iota_{U*}
{\mathfrak p}$ is a coherent analytic sheaf (see Remark \ref{re1}),
\end{itemize}
we have
$$
{\rm degree}(\iota_{U*}{\mathfrak p})\, \leq\, 0~\, ~\,~{\rm (respectively,~
{\rm degree}(\iota_{U*}{\mathfrak p})\, <\, 0)}
$$
(degree is defined in \eqref{dde}).
\end{definition}

\begin{remark}\label{re1}
Let $\iota_U\, :\, U\, \hookrightarrow\, X$ is a dense open subset
such that the complement $X\setminus U$ is a closed complex analytic subset of $X$
of complex codimension at least two, and let $V$ be a holomorphic vector bundle
on $U$. If $X$ is a complex projective manifold, then the direct image
$\iota_{U*} V$ is a coherent analytic sheaf.
\end{remark}

\begin{lemma}\label{lem1}
A pseudo-real principal $G$--bundle $(E_G\, ,\rho)$
over $X$ is semistable if and only if the vector bundle ${\rm ad}(E_G)$
is semistable.

A pseudo-real principal $G$--bundle $(E_G\, ,\rho)$
is semistable if and only if the principal $G$--bundle $E_G$
is semistable.
\end{lemma}

\begin{proof}
If ${\rm ad}(E_G)$ is semistable, then clearly $(E_G\, ,\rho)$ is
semistable.

To prove the converse, assume that ${\rm ad}(E_G)$ is not semistable.
Let
$$
V_1\, \subset\, V_2\, \subset\, \cdots \, \subset\, V_{n-1} \, \subset\,
V_n\,=\, {\rm ad}(E_G)
$$
be the Harder--Narasimhan filtration of ${\rm ad}(E_G)$. Then $n$ is odd,
and $V_{(n+1)/2}$ is a parabolic subalgebra bundle of $\text{ad}(E_G)$ over
a dense open subset $U\, \subset\, X$ such that the complement $X\setminus U\,
\subset\, X$ is a complex analytic subset of complex
codimension at least two (see \cite[p. 216, Lemma 2.11]{AB}).

{}From the uniqueness of the Harder--Narasimhan filtration it follows
immediately that
$$
\widetilde{\rho}(V_{(n+1)/2})\,=\, V_{(n+1)/2}
$$
(see \eqref{rp2} for $\widetilde{\rho}$). Therefore, considering
$V_{(n+1)/2}\, \subset\, {\rm ad}(E_G)$ we conclude that $(E_G\, ,\rho)$
is not semistable.

The vector bundle ${\rm ad}(E_G)$ is semistable if and only if the principal
$G$--bundle $E_G$ is semistable \cite[p. 214, Proposition 2.10]{AB}. Therefore,
the second statement of the lemma follows from the first statement.
\end{proof}

\begin{lemma}\label{lem-s}
Let $(E_G\, ,\rho)$ be a stable pseudo-real principal $G$--bundle
over $X$. Then the vector bundle ${\rm ad}(E_G)$ is polystable.
Also, the principal $G$--bundle $E_G$ is polystable.
\end{lemma}

\begin{proof}
{}From the first part of Lemma \ref{lem1} we know that ${\rm ad}(E_G)$ is
semistable. A semistable sheaf $V$ has a unique maximal polystable subsheaf $F$ with
$$
\text{degree}(V)/\text{rank}(V)\,=\, \text{degree}(F)/
\text{rank}(F)
$$ \cite[page 23, Lemma 1.5.5]{HL}; this
$F$ is called the \textit{socle} of $V$. Assume that ${\rm ad}(E_G)$ is not
polystable. Then there is a unique filtration
\begin{equation}\label{s-f}
0\,=\, F_0\, \subset\,
F_1\, \subset\, F_2\, \subset\, \cdots \, \subset\, F_{n-1} \, \subset\,
F_n\,=\, {\rm ad}(E_G)
\end{equation}
such that for each $i\, \in\, [1\, ,n]$, the quotient $F_i/F_{i-1}$ is
the socle of ${\rm ad}(E_G)/F_{i-1}$. Then $n$ is odd,
and $F_{(n+1)/2}$ is a parabolic subalgebra bundle of $\text{ad}(E_G)$ over
a dense open subset $U\, \subset\, X$ such that the complement $X\setminus U\,
\subset\, X$ is a complex analytic subset of codimension at least
two (see \cite[p. 218]{AB}).

{}From the uniqueness of the filtration in \eqref{s-f} it follows immediately
that $\widetilde{\rho}(F_{(n+1)/2})\,=\, F_{(n+1)/2}$. Therefore, the subsheaf
$F_{(n+1)/2}\, \subset\, {\rm ad}(E_G)$ shows that $(E_G\, ,\rho)$
is not stable. In view of this contradiction, we conclude that 
${\rm ad}(E_G)$ is polystable. 

The second statement of the lemma follows from the first statement and
\cite[p. 224, Corollary 3.8]{AB}.
\end{proof}

\section{Polystable pseudo-real principal bundles and
Einstein-Hermitian connections}

Let $(E_G\, ,\rho)$ be a pseudo-real principal $G$--bundle. Let
$$
{\mathfrak p}\, \subset\, \text{ad}(E_G)
$$
be a parabolic subalgebra bundle such that $\widetilde{\rho}({\mathfrak p})\,=\,
{\mathfrak p}$, where
$\widetilde{\rho}$ is the involution in \eqref{rp2}. Let
$$
R_u({\mathfrak p})\, \subset \, {\mathfrak p}
$$
be the holomorphic subbundle over $X$ whose fiber over any point $x\, \in\, X$ is
the nilpotent radical of the parabolic subalgebra ${\mathfrak p}_x$. Therefore,
the quotient ${\mathfrak p}/R_u({\mathfrak p})$ is a bundle of reductive
Lie algebras. Note that $\widetilde{\rho}(R_u({\mathfrak p}))\,=\,
R_u({\mathfrak p})$.

A \textit{Levi subalgebra bundle} of ${\mathfrak p}$ is a holomorphic subbundle
$$
\ell({\mathfrak p})\, \subset\, {\mathfrak p}
$$
such that for each $x\, \in\, X$, the fiber
$\ell({\mathfrak p})_x$ is a Lie subalgebra of ${\mathfrak p}_x$ with the
composition
$$
\ell({\mathfrak p})\, \hookrightarrow\, {\mathfrak p}\, \longrightarrow\,
{\mathfrak p}/R_u({\mathfrak p})
$$
being an isomorphism, where ${\mathfrak p}\, \longrightarrow\,
{\mathfrak p}/R_u({\mathfrak p})$ is the quotient map.

Let $\ell({\mathfrak p})\, \subset\, {\mathfrak p}$ be a Levi subalgebra bundle
such that $\widetilde{\rho}(\ell({\mathfrak p}))\,=\, \ell({\mathfrak p})$.
Since the fibers of $\ell({\mathfrak p})$ are reductive subalgebras, we may extend
the notion of (semi)stability to $\ell({\mathfrak p})$ as follows.

\begin{definition}\label{d-l-s}
A Levi subalgebra bundle $\ell({\mathfrak p})\, \subset\, {\mathfrak p}$ with
$\widetilde{\rho}(\ell({\mathfrak p}))\,=\, \ell({\mathfrak p})$
is called {\em semistable} (respectively, {\em stable})
if for every pair of the form $(U\, ,{\mathfrak q})$, where
\begin{itemize}
\item $\iota_U\, :\, U\, \hookrightarrow\, X$ is a dense open subset
with $\sigma_X(U)\,=\, U$ such that
the complement $X\setminus U$ is a closed complex analytic subset of $X$ of complex
codimension at least two, and

\item ${\mathfrak q}\, \subsetneq\, \ell({\mathfrak p})\vert_U$ is a parabolic
subalgebra bundle over $U$ such that $\widetilde{\rho}({\mathfrak q})\,=\,
{\mathfrak q}$, and the direct image $\iota_{U*} {\mathfrak q}$ is a coherent
analytic sheaf (see Remark \ref{re1}),
\end{itemize}
we have
$$
{\rm degree}(\iota_{U*}{\mathfrak q})\, \leq\, 0~\, ~\,~{\rm (respectively,~
{\rm degree}(\iota_{U*}{\mathfrak q})\, <\, 0)}\, .
$$
\end{definition}

\begin{definition}\label{def4}
A semistable pseudo-real principal $G$--bundle $(E_G\, ,\rho)$
over $X$ is called {\em polystable} if either $(E_G\, ,\rho)$ is stable,
or there is a proper parabolic subalgebra bundle ${\mathfrak p}\,\subsetneq\,
{\rm ad}(E_G)$, and a Levi subalgebra bundle $\ell({\mathfrak p})\, \subset\,
{\mathfrak p}$, such that the following conditions hold:
\begin{enumerate}
\item $\widetilde{\rho}({\mathfrak p})\, =\, {\mathfrak p}$ and
$\widetilde{\rho}(\ell({\mathfrak p}))\, =\, \ell({\mathfrak p})$, and

\item $\ell({\mathfrak p})$ is stable (see Definition \ref{d-l-s}).
\end{enumerate}
\end{definition}

In Definition \ref{def4}, we start with a semistable pseudo-real principal
bundle to rule out the analogs of direct sum of stable vector bundles of
different slopes.

\begin{lemma}\label{lem2}
Let $(E_G\, ,\rho)$ be a polystable pseudo-real principal $G$--bundle on $X$.
Then the adjoint vector bundle ${\rm ad}(E_G)$ is polystable. Also, the
principal $G$--bundle $E_G$ is polystable.
\end{lemma}

\begin{proof}
If $(E_G\, ,\rho)$ is stable, then it
follows by Lemma \ref{lem-s}. So we assume
that $(E_G\, ,\rho)$ is not stable. From the first part of Lemma \ref{lem1} it
follows that ${\rm ad}(E_G)$ is is semistable. Assume that ${\rm ad}(E_G)$ is not
polystable. Let
$$
F_1\, \subset\, {\rm ad}(E_G)
$$
be the socle (see \eqref{s-f}).

Recalling Definition \ref{def4}, we observe that the vector bundle
$\ell({\mathfrak p})$ in Definition \ref{def4} is polystable with
a proof identical to that of Lemma \ref{lem-s} (this is due to condition (2) in
Definition \ref{def4}). Therefore, we have
\begin{equation}\label{l-c}
\ell({\mathfrak p})\, \subset\, F_1\, .
\end{equation}
But $F_{(n-1)/2}$ in \eqref{s-f} is the nilpotent radical bundle of
the parabolic subalgebra bundle $F_{(n+1)/2}\, \subset\, \text{ad}(E_G)$.
Therefore, all elements of $F_{(n-1)/2}$ are nilpotent.
In particular, all elements of $F_1$ are nilpotent. On the other hand,
$\ell({\mathfrak p})$ is a Levi subalgebra bundle; so for each
$x\, \in\, X$, the fiber $\ell({\mathfrak p})_x$ is a
reductive subalgebra of $\text{ad}(E_G)_x$. Hence \eqref{l-c}
is a contradiction. Therefore, we conclude that ${\rm ad}(E_G)$ is polystable.

The second
statement of the lemma follows from the first statement and
\cite[p. 224, Corollary 3.8]{AB}.
\end{proof}

Consider the semi-direct product $G\rtimes ({\mathbb Z}/2{\mathbb Z})$ defined
by the involution $\sigma_G$ of $G$. So we have a short exact sequence of groups
$$
e\,\longrightarrow\, G \,\longrightarrow\, G\rtimes ({\mathbb Z}/2{\mathbb Z})
\,\longrightarrow\,{\mathbb Z}/2{\mathbb Z}\,\longrightarrow\, e\, .
$$
Take a maximal compact subgroup $\widetilde{K}
\, \subset\, G\rtimes ({\mathbb Z}/2{\mathbb Z})$. Define
\begin{equation}\label{dkg}
K_G\,:=\, \widetilde{K}\cap G\, \subset\, G\, .
\end{equation}
It is a maximal compact subgroup of $G$ which is preserved by $\sigma_G$.

By a \textit{Hermitian structure} on a principal $G$--bundle $E_G$ we will
mean a $C^\infty$ reduction of structure group of $E_G$ to the subgroup $K_G$.
If $E_G$ is holomorphic, and $E_{K_G}\, \subset\, E_G$ is an
Hermitian structure, then there is a unique connection $\nabla$ on $E_{K_G}$
such that the connection on $E_G$ induced by $\nabla$ has the property that
the corresponding $C^\infty$ splitting of the Atiyah exact sequence for
$E_G$ is $\mathbb C$--linear \cite[pp. 191--192, Proposition 5]{At}. This
$\nabla$ is called the \textit{Chern connection} for the reduction
$E_{K_G}$. The connection on $E_G$ induced by $\nabla$ is also
called the Chern connection for the reduction $E_{K_G}$.

Let $E_G$ be a holomorphic principal $G$--bundle, and let $E_{K_G}\, \subset\, E_G$
be an Hermitian structure on $E_G$. The corresponding Chern connection on $E_G$ will
be denoted by $\nabla$; the curvature of $\nabla$ will be denoted by
${\mathcal K}(\nabla)$. Let
$$
\Lambda \, :\, \Omega^{p,q}_X\, \longrightarrow\, \Omega^{p-1,q-1}_X
$$
be the adjoint of the exterior
product with the K\"ahler form $\omega$. The reduction
$E_{K_G}$ is said to be an \textit{Einstein--Hermitian structure} on $E_G$ if
there is an element $\lambda$ in the center of $\mathfrak g$ such that
the section
$$
\Lambda{\mathcal K}(\nabla)\, \in\, C^\infty(X,\, \text{ad}(E_G))
$$
coincides with the one given by $\lambda$ (since the adjoint action of $G$ on the
center of $\mathfrak g$ is trivial, any element of it defines a section of
$\text{ad}(E_G)$).

A principal $G$--bundle $E_G$ admits an Einstein--Hermitian structure if and
only if $E_G$ is polystable, and, moreover, the Einstein--Hermitian connection on a
polystable principal $G$--bundle is unique \cite{Do}, \cite{UY},
\cite[p. 208, Theorem 0.1]{AB}, \cite[p. 24, Theorem 1]{RS}.
Therefore, Lemma \ref{lem2} has the following corollary:

\begin{corollary}\label{c-c}
Let $(E_G\, ,\rho)$ be a polystable pseudo-real principal $G$--bundle.
Then $E_G$ admits an Einstein--Hermitian structure.
\end{corollary}

\begin{assumption}\label{an}
Henceforth, we will always assume that $c\,\in\, Z_{\mathbb R}$ in Definition
\ref{def1} lies in $Z_{\mathbb R}\bigcap K_G$.
\end{assumption}

It was noted in \cite{BH} that without any loss of generality,
the element $c\, \in\, {\mathbb Z}_{\mathbb R}$ in Definition \ref{def1} can be taken
to be of order two (see the end of Section 2.1 of \cite{BH}). But all elements of
$Z_{\mathbb R}$ of order two lie in $Z_{\mathbb R}\bigcap K_G$. Hence
Assumption \ref{an} is not restrictive.

Let $(V\, ,h)$ be a holomorphic Hermitian vector bundle on a complex manifold $M$.
Let $h'$ be another Hermitian structure on $V$. Then there a unique $C^\infty$
endomorphism $A$ of $V$ such that $A^{*_h} \,=\, A$, and
$$
h'(v,w)\, =\, h(x)(\exp(A)(v),w)\, , ~\forall~ x\,\in\, M~\ \text{and}~\ v\, ,
w\, \in\, V_x\, ,
$$ 
where $A^{*_h}$ is the adjoint of $A$ with respect to $h$. Let $\nabla^h$ be
the Chern connection on $V$ for $h$. 

\begin{lemma}\label{lem-c}
The Chern connection on $V$ for $h'$ coincides with $\nabla^h$ if and
only if the above endomorphism $A$ is flat with respect to $\nabla^h$.
\end{lemma}

\begin{proof}
Let $\nabla^{h'}$ be the Chern connection on $V$ for $h'$. Then
$$
\nabla^{h'}-\nabla^h \,=\, \nabla^h (A)
$$
(both sides are $C^\infty$ one-forms with values in $\text{ad}(E_G)$).
\end{proof}

Recall that $\rho$ in Definition \ref{def1} produces an anti-holomorphic
diffeomorphism of $E_G$ which is also denoted by $\rho$ (see \eqref{r-m}).

\begin{proposition}\label{prop-i}
Let $(E_G\, ,\rho)$ be a pseudo-real principal $G$--bundle such that
the principal $G$--bundle $E_G$ is polystable. Then $E_G$
admits an Einstein--Hermitian structure
$$
E_{K_G}\, \subset\, E_G
$$
such that $\rho(E_{K_G})\,=\, E_{K_G}$.
\end{proposition}

\begin{proof}
Let $E_{K_G}\, \subset\, E_G$ be a $C^\infty$ reduction of structure group
of the holomorphic principal $G$--bundle $E_G$ to the subgroup $K_G$. Since
$\sigma_G(K_G)\,=\,K_G$, from \eqref{r-m} it follows immediately that
$\rho(E_{K_G})\,\subset \, E_G$ is also a $C^\infty$ reduction of structure group
to $K_G$. Let $\nabla'$ be a connection on the principal $G$--bundle
$E_G$; it is a $\mathfrak g$--valued one-form on the total space of $E_G$.
Then $(d\sigma_G)\circ\rho^*\nabla'$ is also a connection on $E_G$, where
$d\sigma_G$ is the homomorphism in \eqref{dsg} (recall that $\rho$
is a self-map of the total space of $E_G$). If $\nabla'$ is the Chern
connection for the Hermitian structure $E_{K_G}\, \subset\, E_G$, then it is
straightforward to check that $(d\sigma_G)\circ\rho^*\nabla'$ is the Chern
connection for the Hermitian structure $\rho(E_{K_G})\,\subset \, E_G$.

The principal $G$--bundle $E_G$ admits an Einstein--Hermitian structure,
and the Einstein--Hermitian connection on $E_G$ is unique
(see Corollary \ref{c-c}). Let $\nabla$ denote the Einstein--Hermitian connection
on $E_G$. Since the Einstein--Hermitian connection $\nabla$ is unique,
it follows that $\nabla$ is preserved by $\rho$, meaning $(d\sigma_G)\circ
\rho^*\nabla\,=\, \nabla$. However, the Hermitian structure on $E_G$ giving the
Einstein--Hermitian connection is not unique in general.

Let
$$
E_{K_G}\, \subset\, E_G
$$
be an Hermitian structure on $E_G$ giving the Einstein--Hermitian connection
$\nabla$. Define
$$
E'_{K_G}\,=\, \rho(E_{K_G})\, \subset\, E_G\, .
$$
We noted above that $E'_{K_G}$ is also a $C^\infty$ reduction of structure
group of $E_G$ to $K_G$. Recall from above that the Chern connection on $E_G$
for this Hermitian structure $E'_{K_G}$ coincides with one given by
$\nabla$ using $\rho$. Since $\nabla$ is preserved by $\rho$, the
Chern connection on $E_G$ for $E'_{K_G}$ coincides with $\nabla$.

Let $\mathcal M$ denote the space of all Hermitian structures on $E_G$ that
give the Einstein--Hermitian connection $\nabla$. We note that every Hermitian
structure in $\mathcal M$ is Einstein--Hermitian. If $E_G$ is regularly stable
(meaning $E_G$ is stable and $\text{Aut}(E_G)\,=\, Z_G$),
then ${\mathcal M}\,=\,Z_G/(K_G\cap Z_G)$. Let
\begin{equation}\label{cM}
\rho_{\mathcal M}\, :\, {\mathcal M}\, \longrightarrow\, {\mathcal M}
\end{equation}
be the map defined by $E_{K_G}\, \longmapsto\, \rho(E_{K_G})$ (constructed as above).
Since the element $c$ in Definition \ref{def1} lies in $K_G$
(see Assumption \ref{an}), we conclude that
$\rho_{\mathcal M}$ is an involution. The proposition is equivalent to the
statement that $\rho_{\mathcal M}$ has a fixed point.

Fix a reduction
$$
E^0_{K_G}\, \subset\, E_G
$$
lying in $\mathcal M$. Fix an inner product $h_{\mathfrak g}$ on $\mathfrak g$ which
is invariant under the adjoint action of $K_G$; since $K_G$ is compact, such
an inner product exists. Using the reduction $E^0_{K_G}$, this
$h_{\mathfrak g}$ produces an Hermitian structure on the adjoint vector bundle
$\text{ad}(E_G)$. To see this, note that $\text{ad}(E_G)$ is identified with the
vector bundle $E^0_{K_G}\times^{K_G} \mathfrak g$
associated to $E^0_{K_G}$ for the adjoint action of $K_G$ on $\mathfrak g$.
Therefore, $h_{\mathfrak g}$ induces an Hermitian structure on
$E^0_{K_G}\times^{K_G} \mathfrak g$. So $\text{ad}(E_G)$ gets an Hermitian
structure using its identification with $E^0_{K_G}\times^{K_G} \mathfrak g$. This
Hermitian structure on $\text{ad}(E_G)$ will be denoted by $h_{{\rm ad}(E_G)}$.

Let
\begin{equation}\label{S}
{\mathcal S}\,:=\, \text{ad}(E^0_{K_G})^\perp\, \subset\, \text{ad}(E_G)
\end{equation}
be the orthogonal complement of $\text{ad}(E^0_{K_G})$ with respect to the
Hermitian structure $h_{{\rm ad}(E_G)}$. This orthogonal complement is
in fact independent
of the choice of $h_{\mathfrak g}$. Given any Hermitian structure
$$
E_{K_G}\, \subset\, E_G
$$
on $E_G$, there is a unique $C^\infty$ section $s\, \in\, C^\infty(X,\,
{\mathcal S})$ such that
$$
E_{K_G}\,=\, \exp(s)(E^0_{K_G})
$$
(recall that $\text{ad}(E_G)$
is the Lie algebra bundle associated to $\text{Ad}(E_G)$). Conversely, for any
$$
s\, \in\, C^\infty(X,\, {\mathcal S})\, ,
$$
the image $\exp(s)(E^0_{K_G})\, \subset\, E_G$ is an Hermitian structure on $E_G$.

Let
$$
s_0\, \in\, C^\infty(X,\, {\mathcal S})
$$
be the section such that $\exp(s_0)(E^0_{K_G})\,=\, \rho_{\mathcal M}(E^0_{K_G})$,
where $\rho_{\mathcal M}$ is constructed in \eqref{cM}.

Let $\nabla^{\rm ad}$ be the connection on the vector bundle $\text{ad}(E_G)$
induced by the Einstein--Hermitian connection $\nabla$. From
Lemma \ref{lem-c} it can be deduced that $s_0$ is covariant constant
(flat) with respect to $\nabla^{\rm ad}$. To prove this, take any faithful
holomorphic representation $G\, \hookrightarrow\, \text{GL}(W)$. Fix a maximal
compact subgroup of $\text{GL}(W)$ containing $K_G$. Consider the two
Hermitian structures on the associated vector bundle $E_G\times^G W$ given
by $E^0_{K_G}$ and $\rho_{\mathcal M}(E^0_{K_G})$. Since their Chern
connections coincide, using Lemma \ref{lem-c} we deduce that $s_0$ is flat
with respect to $\nabla^{\rm ad}$.

We will prove that that the Hermitian structure $\exp(s_0/2)(E^0_{K_G})$
on $E_G$ is fixed by $\rho_{\mathcal M}$.

To prove that $\exp(s_0/2)(E^0_{K_G})$ lies in $\mathcal M$, note that
$s_0/2$ is flat with respect to $\nabla^{\rm ad}$ because $s_0$ is so. Therefore, using
Lemma \ref{lem-c} we conclude that the Chern connection for the Hermitian
structure $\exp(s_0/2)(E^0_{K_G})$ coincides with $\nabla$ (as before, take a
faithful holomorphic representation $G$ and apply Lemma \ref{lem-c} to the
associated vector bundle). Therefore,
$$
\exp(s_0/2)(E^0_{K_G})\, \in\, \mathcal M\, .
$$

Take any point $x\, \in\, X$. Fix a point
$$
z_0\, \in\, (E^0_{K_G})_x\, .
$$
Identify $(E^0_{K_G})_x$ and $(E_G)_x$ with $K_G$ and $G$ respectively by
sending any element $z_0g$ to $g$. The space of all reductions of the structure
group of the principal $G$--bundle $(E_G)_x\,
\longrightarrow\, \{x\}$ to the subgroup $K_G$ is identified with $(E_G)_x/K_G$. Hence
using the above identification of $(E_G)_x$ with $G$, this space of reductions
coincides with $G/K_G$.

Let $g_0\, \in\, G$ be the unique element such that
\begin{equation}\label{g0}
\exp(s_0)(x)(z_0)\,=\, z_0g_0\, .
\end{equation}
For the element $g_0K_G\,\in\, G/K_G$,
$$
g_0K_G\,=\, (\rho_{\mathcal M}(E^0_{K_G}))_x\,=\, \rho((E^0_{K_G})_{\sigma_X(x)})
\, \subset\, (E_G)_x
$$
using the above identification between $G/K_G$ and the space of all reductions of
the principal $G$--bundle $(E_G)_x\,\longrightarrow\, \{x\}$ to the subgroup $K_G$.

We note that using $z_0$, the fiber $\text{ad}(E_G)$ is identified with the 
Lie algebra $\mathfrak g$. This identification sends any $v\, \in\, \mathfrak g$
to the equivalence class of $(z_0\, ,v)$ (recall that the total space of
$\text{ad}(E_G)$ is a quotient of $E_G\times \mathfrak g$).
Let
$$
v_0\, \in \, {\mathfrak g}
$$
be the element given by $s_0(x)\,\in\, \text{ad}(E_G)_x$ using this
identification. From \eqref{g0} we have
\begin{equation}\label{g-v}
\exp (v_0)\,=\, g_0\, .
\end{equation}

Next we show that any reduction $E'_{K_G}\, \subset\, E_G$ lying in
$\mathcal M$ is uniquely determined by its restriction $(E'_{K_G})_x\,
\subset\, (E_G)_x$. To prove this, recall that
the Chern connection on $E_G$ for $E'_{K_G}$ coincides with $\nabla$. Hence we
can reconstruct $E'_{K_G}$ from $(E'_{K_G})_x$ by taking parallel translations
of $(E'_{K_G})_x\, \subset\, (E_G)_x$ using $\nabla$. Hence $E'_{K_G}$ is
uniquely determined by $(E'_{K_G})_x$.

Let
$$
{\mathcal M}^x\, \subset\, G/K_G
$$
be the image of the map ${\mathcal M}\, \longrightarrow\, G/K_G$
that sends any $E'_{K_G}\, \subset\, E_G$ in ${\mathcal M}$ to
the reduction $(E'_{K_G})_x\, \subset\, (E_G)_x$ (recall that the
space of all reductions of the principal $G$--bundle $(E_G)_x\,
\longrightarrow\,\{x\}$ to the subgroup $K_G$ is identified with $G/K_G$).
Since any reduction $E'_{K_G}\, \subset\, E_G$ lying in
$\mathcal M$ is uniquely determined by its restriction $(E'_{K_G})_x\,
\subset\, (E_G)_x$, the map $\rho_{\mathcal M}$ in \eqref{cM} produces a map
\begin{equation}\label{wrm}
\widetilde{\rho}^x_{\mathcal M}\, :\, {\mathcal M}^x\, \longrightarrow\,
{\mathcal M}^x\, .
\end{equation}

Using \eqref{r-m} it follows that $\widetilde{\rho}^x_{\mathcal M}$ is the
restriction of the map
\begin{equation}\label{fg0}
f_{g_0}\, :\, G/K_G\, \longrightarrow\, G/K_G\, ,~\ gK_G\, \longmapsto\,
g_0\sigma_G(g)K_G\, ,
\end{equation}
where $g_0$ is the element of $G$ in \eqref{g0}.

The direct sum of the Killing form on $[{\mathfrak g}\, ,
{\mathfrak g}]$ and an inner product
on the center of $\mathfrak g$ is a nondegenerate $G$--invariant form
on $\mathfrak g$. This form produces a Riemannian metric on $G/K_G$. The
map $f_{g_0}$ in \eqref{fg0} is an isometry with respect to this
Riemannian metric. Given any two points of
$G/K_G$, there is a unique geodesic passing through them.

The map $\widetilde{\rho}^x_{\mathcal M}$ in \eqref{wrm} interchanges the two points
$(E^0_{K_G})_x$ and $(\rho_{\mathcal M}(E^0_{K_G}))_x$ of ${\mathcal M}^x$.
Since $\widetilde{\rho}^x_{\mathcal M}$ is the restriction of the isometry
$f_{g_0}$, the mid-point of the unique geodesic between the two points
$(E^0_{K_G})_x$ and $(\rho_{\mathcal M}(E^0_{K_G}))_x$ is fixed by
$\widetilde{\rho}^x_{\mathcal M}$, provided this mid-point lies in 
${\mathcal M}^x$.

The earlier identification between $G/K_G$ and the space of all reductions
of the principal $G$--bundle $(E_G)_x\, \longrightarrow\, \{x\}$ to $K_G$
(given by $z_0$) sends the reduction $(E^0_{K_G})_x$ (respectively,
$(\rho_{\mathcal M}(E^0_{K_G}))_x$) to $eK_G$ (respectively, $g_0K_G$).
The mid-point of the unique geodesic in $G/K_G$ between $eK_G$ and $g_0K_G$ is
$\exp(v_0/2)K_G$ (see \eqref{g-v}). Therefore, the mid-point of the unique
geodesic between the two points $(E^0_{K_G})_x$ and $(\rho_{\mathcal M}
(E^0_{K_G}))_x$ is $(\exp(s_0/2)(E^0_{K_G}))_x$.

We have shown above that $\exp(s_0/2)(E^0_{K_G})$ lies in $\mathcal M$. Consequently,
for every point $x\, \in\, X$, the reduction
$$
(\exp(s_0/2)(E^0_{K_G}))_x\, \subset\, (E_G)_x
$$
coincides with $(\rho_{\mathcal M}(\exp(s_0/2)(E^0_{K_G})))_x\, \subset\,
(E_G)_x$. Therefore, the Hermitian structure $\exp(s_0/2)(E^0_{K_G})$
on $E_G$ is fixed by $\rho_{\mathcal M}$.
\end{proof}

Lemma \ref{lem2} and Proposition \ref{prop-i} together give the following:

\begin{corollary}\label{cor1}
Let $(E_G\, ,\rho)$ be a polystable pseudo-real principal $G$--bundle.
Then $E_G$ admits an Einstein--Hermitian structure $E_{K_G}\, \subset\, E_G$
such that $\rho(E_{K_G})\,=\, E_{K_G}$.
\end{corollary}

\begin{proposition}\label{prop3}
Let $(E_G\, ,\rho)$ be a pseudo-real principal $G$--bundle admitting
an Einstein--Hermitian structure $E_{K_G}\, \subset\, E_G$
such that $\rho(E_{K_G})\,=\, E_{K_G}$. Then $(E_G\, ,\rho)$ is polystable.
\end{proposition}

\begin{proof}
As before, $\nabla^{\rm ad}$ is the connection on $\text{ad}(E_G)$ induced by the
Einstein--Hermitian connection on $E_G$. This connection $\nabla^{\rm ad}$
is clearly Einstein--Hermitian. Therefore, $\text{ad}(E_G)$ is
polystable, in particular, it is semistable.
Hence the pseudo-real principal $G$--bundle $(E_G\, ,\rho)$ is semistable
(see Lemma \ref{lem1}). If $(E_G\, ,\rho)$ is stable, then $(E_G\, ,\rho)$ is
polystable. Therefore, assume that $(E_G\, ,\rho)$ is not stable.

Take a pair $(U\, ,{\mathfrak p})$ as in Definition \ref{def3}
such that
$$
{\rm degree}(\iota_{U*}{\mathfrak p})\,=\, 0\, .
$$
Since $\text{ad}(E_G)$ is polystable of degree zero, the subbundle
${\mathfrak p}$ of $\text{ad}(E_G)\vert_U$ extends to a subbundle of
$\text{ad}(E_G)$ over $X$. To see this write, $\text{ad}(E_G)$ as a direct
sum of stable vector bundles. The statement is clear for a stable vector bundle;
the statement for polystable case follows from this. 
This extended vector bundle will be denoted by
${\mathfrak p}'$. Clearly, ${\mathfrak p}'$ is a parabolic
subalgebra bundle of $\text{ad}(E_G)$. We also have $\widetilde{\rho}({\mathfrak p}')
\,=\, {\mathfrak p}'$, because $\widetilde{\rho}({\mathfrak p})
\,=\, {\mathfrak p}$. Furthermore,
$$
{\rm degree}({\mathfrak p}')\,=\, {\rm degree}(\iota_{U*}{\mathfrak p})\,=\, 0\, .
$$

Let
$$
{\mathfrak p}\, \subset\, \text{ad}(E_G)
$$
be a smallest parabolic subalgebra bundle over $X$ such that
\begin{itemize}
\item $\widetilde{\rho}({\mathfrak p})\, =\, {\mathfrak p}$, and

\item ${\rm degree}({\mathfrak p})\,=\, 0$.
\end{itemize}
It should be clarified that ${\mathfrak p}$ need not be unique.

We will show that the connection $\nabla^{\rm ad}$ on $\text{ad}(E_G)$
preserves the subbundle ${\mathfrak p}$.

The vector bundle $\text{ad}(E_G)$ is polystable of degree zero.
Since ${\rm degree}({\mathfrak p})\,=\, 0$, there is a holomorphic subbundle
$W\, \subset\, \text{ad}(E_G)$ such that the natural homomorphism
$$
{\mathfrak p}\oplus W\, \longrightarrow\, \text{ad}(E_G)
$$
is an isomorphism. Hence both ${\mathfrak p}$ and $W$ are of polystable
of degree zero. Therefore, from the uniqueness of the Einstein--Hermitian connection
it follows that the Einstein--Hermitian connection $\nabla^{\rm ad}$ is the
direct sum of the Einstein--Hermitian connections on ${\mathfrak p}$ and $W$.
In particular, the connection $\nabla^{\rm ad}$ preserves the subbundle ${\mathfrak p}$.

The adjoint vector bundle $\text{ad}(E_{K_G})$ is a totally real subbundle of
$\text{ad}(E_G)$, meaning $\text{ad}(E_{K_G})\bigcap
\sqrt{-1}\cdot\text{ad}(E_{K_G})\,=\, 0$. Since
both the subbundles ${\mathfrak p}$ and
$\text{ad}(E_{K_G})$ are preserved by $\nabla^{\rm ad}$, it follows that
${\mathfrak p}\bigcap \text{ad}(E_{K_G}$ is
a real subbundle of $\text{ad}(E_G)$ preserved by $\nabla^{\rm ad}$.
Consider the complexified vector bundle
$$
{\mathcal E}\, :=\, ({\mathfrak p}\cap \text{ad}(E_{K_G}))\otimes_{\mathbb R}
\mathbb C\, .
$$
Since $\text{ad}(E_{K_G})$ is a totally real subbundle, this $\mathcal E$
is a complex subbundle of $\text{ad}(E_{K_G})$; it is clearly preserved by
$\nabla^{\rm ad}$. In particular, ${\mathcal E}$ is a holomorphic subbundle
of $\mathfrak p$. This holomorphic subbundle ${\mathcal E}\, \subset\, \mathfrak p$
is a Levi subalgebra bundle of $\mathfrak p$.

The given condition that $\rho(E_{K_G})\,=\, E_{K_G}$ implies that
$\widetilde{\rho}(\text{ad}(E_{K_G}))\,=\, \text{ad}(E_{K_G})$. Since we also
have $\widetilde{\rho}({\mathfrak p})\, =\, {\mathfrak p}$, it follows
immediately that
$$
\widetilde{\rho}({\mathcal E})\,=\, {\mathcal E}\, .
$$

{}From the minimality assumption on ${\mathfrak p}$ it can be deduced that
the Levi subalgebra bundle ${\mathcal E}$ is stable. To see this, assume that
${\mathfrak q}\, \subset\, {\mathcal E}\vert_U$ is a parabolic subalgebra bundle
violating the stability of the Levi subalgebra bundle ${\mathcal E}$. Then the
direct sum ${\mathfrak q}\oplus R_n({\mathfrak p})$, where $R_n({\mathfrak p})\,\subset
\, {\mathfrak p}\vert_U$ is the nilpotent radical, is property contained in
${\mathfrak p}$, and it contradicts the minimality assumption on ${\mathfrak p}$.
Hence we conclude that the Levi subalgebra bundle $\mathcal E$ is stable. Consequently,
$(E_G\, ,\rho)$ is polystable.
\end{proof}

Proposition \ref{prop-i} and Proposition \ref{prop3} together give the following:

\begin{corollary}\label{cor3}
If $(E_G\, ,\rho)$ is a pseudo-real principal $G$--bundle such that
the holomorphic principal $G$--bundle $E_G$ is polystable. Then
$(E_G\, ,\rho)$ is polystable.
\end{corollary}

\section{Representations of the extended fundamental group in a compact
subgroup}\label{sec4}

Fix a point $x_0\, \in\,X$ such that $\sigma_X(x_0)\,\not=\, x_0$.
Let
$$
\Gamma(x_0)\, =\, \Gamma(X,x_0)
$$
be the homotopy classes of paths $\gamma\, :\, [0\, ,1]\, \longrightarrow\, X$
such that $\gamma(0)\,=\, x_0$ and $\gamma(1)\,\in\, \{x_0\, ,\sigma_X(x_0)\}$.
Take two paths $\gamma_1\, ,\gamma_2\, \in\, \Gamma(x_0)$. If $\gamma_2(1)\, =\,
x_0$, then define $\gamma_2\cdot\gamma_1\,=\, \gamma_1\circ\gamma_2$, where
``$\circ$'' denotes composition of paths. If $\gamma_2(1)\, =\, \sigma_X(x_0)$,
then define $\gamma_2\cdot\gamma_1\,=\, \sigma_X(\gamma_1)\circ\gamma_2$.
These operations make $\Gamma(x_0)$ into a group (see \cite{BHH}). The inverse
of $\gamma\,\in\,\Gamma(x_0)$ with $\gamma(1)\, =\, \sigma_X(x_0)$ is
represented by the path $t\, \longmapsto\, \sigma_X(\gamma(1-t))$. This
group $\Gamma(x_0)$ fits in a short exact sequence of groups
\begin{equation}\label{fg}
e\,\longrightarrow\, \pi_1(X,x_0)\,\longrightarrow\,\Gamma(x_0)\,
\stackrel{\eta}{\longrightarrow}\,
{\mathbb Z}/2{\mathbb Z}\,\longrightarrow\, e\, ,
\end{equation}
where $\eta(\gamma)\,=\, 0$ if $\gamma(1)\,=\, x_0$, and
$\eta(\gamma)\,=\, 1$ if $\gamma(1)\,=\, \sigma_X(x_0)$. If there is a
point $y\, \in\, Y$ such that $\sigma_X(y)\,=\, y$, then \eqref{fg} is
a right-split (the exact sequence is isomorphic to a semi-direct product).
To see this, fix a path $\gamma_0$ from $x_0$ to $y$. Then the composition
$\gamma_1\, :=\, \sigma_X(\gamma_0)^{-1}\circ\gamma_0\, \in\, \eta^{-1}(1)$
is of order two. So $1\, \longmapsto\, \gamma_1$ is a right-splitting
of \eqref{fg}.

Let $K_G$ be the maximal compact subgroup of $G$ defined earlier (see \eqref{dkg}).
The group $\widetilde K$ in \eqref{dkg} is identified with the semi-direct product
$K_G\rtimes ({\mathbb Z}/2{\mathbb Z})$ for the involution $\sigma_G$ of $K_G$.
In particular, the set $\widetilde K$ is identified with the set $K_G\times
\{0\, ,1\}$.

Let $\text{Map}'(\Gamma(x_0)\, , {\widetilde K})$ be the space of all maps
$$
\delta\, :\, \Gamma(x_0)\,\longrightarrow\, {\widetilde K}
$$
such that the following diagram is commutative:
\begin{equation}\label{di}
\begin{matrix}
e & \longrightarrow & \pi_1(X,x_0)&\longrightarrow &\Gamma(x_0) &
\stackrel{\eta}{\longrightarrow} &
{\mathbb Z}/2{\mathbb Z} &\longrightarrow & e\\
&& \Big\downarrow && ~\Big\downarrow\delta && \Vert\\
e & \longrightarrow & K_G &\longrightarrow &{\widetilde K} &
\stackrel{\eta'}{\longrightarrow} &
{\mathbb Z}/2{\mathbb Z} &\longrightarrow & e
\end{matrix}
\end{equation}

We write ${\mathbb Z}/2{\mathbb Z}\,=\, \{0\, ,1\}$.
For any $c\, \in\, Z_{\mathbb R}\bigcap K_G$, let
$\text{Hom}_c(\Gamma(x_0)\, , {\widetilde K})$ be the space of all maps
$$
\delta\, \in\, \text{Map}'(\Gamma(x_0)\, , {\widetilde K})
$$
such that
\begin{itemize}
\item the restriction of $\delta$ to $\pi_1(X,x_0)$ is a homomorphism of groups,

\item $\delta(g'g)\,=\, c\delta(g')\delta(g)$, if $\eta(g)\, =\,1\,=\,
\eta(g')$ (the homomorphism $\eta$ is defined in \eqref{fg}), and

\item $\delta(g'g)\,=\, \delta(g')\delta(g)$ otherwise (meaning
if $\eta(g)\cdot \eta(g')\, =\,0$).
\end{itemize}

We note that if $c\,=\, e$, then $\text{Hom}_c(\Gamma(x_0)\, , {\widetilde K})$
is the space of all homomorphisms from $\Gamma(x_0)$ to ${\widetilde K}$
satisfying \eqref{di}.

Take any $\delta\, \in\, \text{Hom}_c(\Gamma(x_0)\, , {\widetilde K})$. We will
construct from $\delta$ a polystable pseudo-real principal $G$--bundle on $X$.

Consider the restriction $\delta'\, :=\, \delta\vert_{\pi_1(X,x_0)}$ (see
\eqref{fg}). It is a homomorphism from $\pi_1(X,x_0)$ to $K_G$.
Therefore, $\delta'$ gives
\begin{itemize}
\item a principal $K_G$--bundle $E_{K_G}$ equipped with a flat
connection $\nabla^K$, and

\item a base point $z_0\, \in\, (E_{K_G})_{x_0}$ over the base point $x_0$.
\end{itemize}

Let $E_G\, :=\, E_{K_G}\times^{K_G} G\, \longrightarrow\, X$ be the principal
$G$--bundle obtained by extending the structure group of $E_{K_G}$ using the
inclusion of $K_G$ in $G$. The flat connection $\nabla^K$ defines a holomorphic
structure on $E_G$. This holomorphic principal $G$--bundle $E_G$ is polystable
because $\nabla^K$ is a flat Hermitian connection.

We will construct a diffeomorphism
\begin{equation}\label{rt}
\rho_{\sigma_X(x_0)}\, :\, (E_G)_{\sigma_X(x_0)}\, \longrightarrow\, (E_G)_{x_0}
\end{equation}
between the fibers of $E_G$. For that, take any $\gamma\, \in\, \Gamma(x_0)$ such
that $\eta(\gamma)\,=\, 1$ (see \eqref{fg} for $\eta$). Let $g_\gamma\, \in\, K_G$
be the element such that the canonical identification of the set $\widetilde K$
with $K_G\rtimes \{0\, ,1\}$ takes $\delta(\gamma)$ to $(g_\gamma\, ,1)$. Let
$$
z'_0\, \in\, (E_G)_{\sigma_X(x_0)}
$$
be the element obtained by the parallel translation of the base point $z_0$
along $\gamma$ for the connection $\nabla^K$.
The map $\rho_{\sigma_X(x_0)}$ in \eqref{rt} is defined as follows:
$$
\rho_{\sigma_X(x_0)}(z'_0 g)\,=\, z_0\sigma_G(g^{-1}_\gamma g) \, \in\,
(E_G)_{x_0}\, , ~\ g\, \in\, G\, .
$$

\begin{lemma}\label{le1}
The map $\rho_{\sigma_X(x_0)}$ defined above is independent of the choice of
$\gamma$.
\end{lemma}

\begin{proof}
Take an element $\gamma_1\, \in\, \pi_1(X,x_0)$, and replace $\gamma$ by
the element $\gamma_1\gamma\, \in\, \Gamma(x_0)$
represented by the path $\gamma\circ\gamma_1$.
Let $g_{\gamma_1\gamma}$ be the element of $K_G$
such that $$\delta(\gamma_1\gamma)\,=\, (g_{\gamma_1\gamma}\, ,1)\, .$$
Then $g_{\gamma_1\gamma}\,=\,\delta(\gamma_1)g_\gamma$.
The element $z'_0$ gets replaced by $z'_0\delta(\gamma_1)^{-1}$. Therefore,
the map $\rho_{\sigma_X(x_0)}$ constructed as above using $\gamma_1\gamma$
in place of $\gamma$ sends the point $z'_0\delta(\gamma_1)^{-1}$ to
$z_0\sigma_G(g_\gamma^{-1})\sigma_G(\delta(\gamma_1))^{-1}$.

Consequently, the two maps $\rho_{\sigma_X(x_0)}$ constructed using
$\gamma$ and $\gamma_1\gamma$ respectively
coincide on the point $z'_0\delta(\gamma_1)^{-1}$. On the
other hand, both these maps satisfy the condition that
\begin{equation}\label{id}
\rho_{\sigma_X(x_0)}(yh)\,=\, \rho_{\sigma_X(x_0)}(h)\sigma_G(h)
\end{equation}
for all $y\, \in\, (E_G)_{\sigma_X(x_0)}$ and $h\, \in\, G$. These together
imply the two maps coincide on the entire $(E_G)_{\sigma_X(x_0)}$. Therefore,
the map $\rho_{\sigma_X(x_0)}$ is independent of the choice of $\gamma$.
\end{proof}

The map $\rho_{\sigma_X(x_0)}$ is clearly anti-holomorphic.

We will now show that $\rho_{\sigma_X(x_0)}$ is independent of the
base point $z_0$.

Take any $g_0\,\in\, K_G$. Define
$$
\widetilde{\delta}\, :\, \Gamma(x_0)\, \longrightarrow\, \widetilde{K}\, ,~\,~\, ~
z\, \longmapsto\, g^{-1}_0\delta(z) g_0
$$
(recall that $K_G$ is a subgroup of $\widetilde{K}$). Note that $\widetilde{\delta}
\,\in\, \text{Hom}_c(\Gamma(x_0)\, , {\widetilde K})$. If we replace $\delta$ by
$\widetilde{\delta}$, then the flat principal $E_K$--bundle $(E_K\, ,\nabla^K)$
remains unchanged, but the base point $z_0$ gets replaced by $z_0g_0$.

\begin{lemma}\label{le3}
The map $\rho_{\sigma_X(x_0)}$ in \eqref{rt} for $\delta$ coincides with the
corresponding map for $\widetilde{\delta}$. In other words, $\rho_{\sigma_X(x_0)}$
does not change if $\delta$ is conjugated by an element of $K_G$.
\end{lemma}

\begin{proof}
Take the element $\gamma\, \in\, \Gamma(x_0)$ in the construction of the map in
\eqref{rt}. Replace $\delta$ by $\widetilde{\delta}$. Then
$z_0$ gets replaced by $z_0g_0$, and hence $z'_0$ gets
replaced by $z'_0g_0$. The element $g_\gamma$ gets replaced
by $g^{-1}_0g_\gamma\sigma_G(g_0)$. Therefore, the two maps constructed
as in \eqref{rt} for $\widetilde{\delta}$ and $\delta$ respectively
coincide at the point $z'_0g_0$. Now from \eqref{id} we conclude that
the two maps coincide on entire $(E_G)_{x_0}$.
\end{proof}

Take a point $x_1\, \in\, X$. If $\sigma_X(x_1)\, \not=\, x_1$, then
define $\Gamma(x_1)\,=\, \Gamma(X,\, x_1)$ as before by replacing $x_0$
with $x_1$. If $\sigma_X(x_1)\, =\, x_1$, then define $\Gamma(x_1)$
to be the semi-direct product
$$
\Gamma(x_1)\, :=\, \pi_1(X,x_1)\rtimes ({\mathbb Z}/2{\mathbb Z})
$$
constructed using the involution of $\pi_1(X,x_1)$ given by $\sigma_X$.

Fix a path $\gamma_0$ in $X$ from $x_1$ to $x_0$.
Then we have an isomorphism $\pi_1(X,x_0)\, \longrightarrow\,
\pi_1(X,x_1)$ defined by $\gamma\, \longmapsto\, \gamma^{-1}_0\circ\gamma
\circ\gamma_0$ (as before, ``$\circ$'' is composition of paths).
This isomorphism extends to an isomorphism $\Gamma(x_0)\, \longrightarrow\,
\Gamma(x_1)$ by sending any $\gamma\, \in\, \eta^{-1}(1)$ to
$\sigma_X(\gamma^{-1}_0)\circ\gamma\circ \gamma_0$. The inverse of this
isomorphism $\Gamma(x_0)\, \longrightarrow\,
\Gamma(x_1)$ produces a bijection
$$
\beta\, :\, \text{Hom}_c(\Gamma(x_0)\, , {\widetilde K})\, \longrightarrow\,
\text{Hom}_c(\Gamma(x_1)\, , {\widetilde K})
$$
by composition of maps.
The flat principal $K$--bundle corresponding to any
$$
\delta\, \in \,\text{Hom}_c(\Gamma(x_0)\, , {\widetilde K})
$$
is identified with the flat principal $K$--bundle corresponding to
$\beta(\delta)$; the base point in the bundle changes by parallel
translation along $\gamma_0$.

{}From Lemma \ref{le3} it can be deduced that the isomorphism
$$
\rho_{\sigma_X(x_1)}\, :\, (E_G)_{\sigma_X(x_1)}\, \longrightarrow\, (E_G)_{x_1}
$$
constructed as in \eqref{rt} for $\beta(\delta)$ is independent of the choice
of $\gamma_0$. Indeed, for two choices of $\gamma_0$, the corresponding
isomorphisms $\Gamma(x_0)\, \longrightarrow\, \Gamma(x_1)$ differ by an inner
automorphism of $\Gamma(x_0)$ given by an element of $\pi_1(X,x_0)$.
Therefore, for two choices
of $\gamma_0$, the corresponding bijections from $\text{Hom}_c(\Gamma(x_0)\, ,
{\widetilde K})$ to $\text{Hom}_c(\Gamma(x_1)\, , {\widetilde K})$
differ by an inner automorphism of ${\widetilde K}$ by an element of $K_G$. By
Lemma \ref{le3}, an inner automorphism of ${\widetilde K}$ by an element of $K_G$
does not affect the map in \eqref{rt}.

Therefore, we get a map
$$
\rho_X\, :\, E_G\, \longrightarrow\, E_G
$$
by running the base point $x_1$ over entire $X$. From the construction
of $\rho_X$ it follows immediately that
\begin{itemize}
\item $\rho_X(zg)\, =\, \rho_X(z)\sigma_G(g)$ for all $z\, \in\, E_G$ and
$g\, \in\, G$, and 

\item $\rho_X$ is anti-holomorphic.
\end{itemize}

Let
$$
\rho\, :\, E_G\, \longrightarrow\, \sigma^*_X \overline{E}_G
$$
be the map given by $\rho_X$ and the natural identification of the total
spaces of $E_G$ and $\sigma^*_X \overline{E}_G$. From the above two properties
of $\rho_X$ it follows immediately that $\rho$ is a holomorphic isomorphism
of principal $G$--bundles.

\begin{proposition}\label{prop4}
The pair $(E_G\, ,\rho)$ constructed above from $\delta\, \in\,{\rm Hom}_c
(\Gamma(x_0), \widetilde{K})$ is a pseudo-real principal $G$--bundle such that
the corresponding element in $Z_{\mathbb R}$ (see Definition \ref{def1}) is $c$.
\end{proposition}

\begin{proof}
To prove the proposition it suffices to show that the composition
$$
(E_G)_{\sigma_X(x_0)}\, \stackrel{\rho_{\sigma_X(x_0)}}{\longrightarrow}\,
(E_G)_{x_0}\, \stackrel{\rho_{x_0}}{\longrightarrow}\, (E_G)_{\sigma_X(x_0)}
$$
is multiplication by $c$.

Fix a path $\gamma$ in $X$ from $x_0$ to $\sigma_X(x_0)$. So $\gamma\, \in\,
\eta^{-1}(1)\, \subset\, \Gamma(x_0)$. As before, $z_0$ is the base point in
$(E_G)_{x_0}$. Let $z'_0\, \in\, (E_G)_{\sigma_X(x_0)}$ be the point obtained
by the parallel translation of $z_0$ along $\gamma$. We will identify
$\Gamma(x_0)$ with $\Gamma(x_1)$ using the reverse path
$\gamma'\, :\, [0\, ,1]\, \longrightarrow\, X$ from $\sigma_X(x_0)$ from $x_0$
defined by $\gamma'(t)\,=\, \gamma (1-t)$. Let
$$
\delta'\, \in\, \text{Hom}_c(\Gamma (x_1), \widetilde{K})
$$
be the element given by $\delta$ using this isomorphism of $\Gamma(x_0)$ with
$\Gamma(x_1)$. The base point in $(E_G)_{\sigma_X(x_0)}$ for $\delta'$ is
$z'_0$.

We will use the path $\gamma$ to construct $\rho_{\sigma_X(x_0)}$, and
we will use
the path $\sigma_X(\gamma)$ to construct $\rho_{x_0}$. Although these maps are
independent of the choice of path (see Lemma \ref{le1}), we need to fix paths
for explicit computations.

As before, $g_\gamma\, \in\, K_G$ is such that the canonical identification
of $\Gamma(x_0)$ with $K_G\rtimes \{0\, ,1\}$ takes $\delta(\gamma)$ to $(g_\gamma\, ,1)$.

We have
\begin{equation}\label{z1}
\rho_{\sigma_X(x_0)}(z'_0)\,=\, z_0\sigma_G(g_\gamma)^{-1}\, .
\end{equation}
The parallel translation along the path $\sigma_X(\gamma)$ takes $z'_0$ to
$z_0\delta(\gamma\gamma)^{-1}$ (the element $\gamma\gamma\, \in\, \pi_1(X,x_0)$
is given by the composition $\sigma_X(\gamma)\circ\gamma$). Therefore,
$$
\rho_{x_0}(z_0\delta(\gamma\gamma)^{-1})\,=\, z'_0\sigma_G(g^{-1}_\gamma)\, ;
$$
this uses the fact that the above isomorphism between $\Gamma(x_0)$ and
$\Gamma(x_1)$
takes $\gamma\, \in\,\eta^{-1}(1)\, \subset\, \Gamma(x_0)$ to the homotopy class
of $\sigma_X(\gamma)$. Therefore, substituting $\sigma_X(x_0)$ in place of $x_0$ in the
identity \eqref{id}, we get
\begin{equation}\label{z2}
\rho_{x_0}(z_0\sigma_G(g_\gamma)^{-1})\,=\, z'_0\sigma_G(g^{-1}_\gamma)
\sigma_G(\delta(\gamma\gamma)\sigma_G(g_\gamma)^{-1})\, .
\end{equation}

But $\delta(\gamma\gamma)\,=\, \delta(\gamma)^2 c\,=\, g_\gamma\sigma_G(g_\gamma)c$.
Hence
$$
\sigma_G(g^{-1}_\gamma)
\sigma_G(\delta(\gamma\gamma)\sigma_G(g_\gamma)^{-1})
\,=\, \sigma_G(g^{-1}_\gamma)\sigma_G(g_\gamma)g_\gamma (g_\gamma)^{-1} c
\,=\, c\, .
$$
Therefore, from \eqref{z2} we have
$$
\rho_{x_0}(z_0\sigma_G(g_\gamma)^{-1})\,=\,z'_0 c\, .
$$
Combining this with \eqref{z1}, we conclude that
\begin{equation}\label{z3}
\rho_{x_0}\circ \rho_{\sigma_X(x_0)}(z'_0)\, =\, z'_0c\, .
\end{equation}

{}From \eqref{id} it follows that $\rho_{x_0}\circ \rho_{\sigma_X(x_0)}$ commutes
the action of $G$ on $(E_G)_{\sigma_X(x_0)}$. Therefore, from \eqref{z3} we conclude
that $\rho_{x_0}\circ\rho_{\sigma_X(x_0)}$ coincides with
multiplication by $c$.
\end{proof}

We noted earlier that the holomorphic principal $G$--bundle $E_G$ is polystable.
Therefore, from Corollary \ref{cor3} it follows that the pseudo-real principal
$G$--bundle $(E_G\, ,\rho)$ is polystable.

Since $E_G$ admits a flat connection, it follows that all the rational characteristic
classes of $E_G$ of positive degree vanish.

We will now describe a reverse construction.

Let $(E_G\, ,\rho)$ be a polystable pseudo-real principal $G$--bundle
such that the corresponding element
in $Z_{\mathbb R}$ (see Definition \ref{def1}) is $c\, \in\,
Z_{\mathbb R}\bigcap K_G$.

Assume that the following two conditions hold:
\begin{itemize}
\item the second Chern class of ${\rm ad}(E_G)$ satisfies the condition
$$
\int_X c_2({\rm ad}(E_G))\wedge\omega^{\dim_{\mathbb C}(X) -2}\,=\, 0\, ,
$$
and

\item for any character $\chi$ of $G$, the line bundle over $X$ associated to
$E_G$ for $\chi$ is of degree zero.
\end{itemize}
These two numerical conditions together imply
that the Einstein--Hermitian connection on $E_G$ is flat
\cite[p. 115, Lemma 4.12]{Ko}; in \cite{Ko}, this is proved for vector bundles,
but it extends to principal $G$--bundles by taking vector bundles associated
to irreducible representations of $G$. Therefore, these
numerical conditions imply that
all the rational characteristic classes of $E_G$ of positive degree vanish.

The Einstein--Hermitian connection on $E_G$ will be denoted by $\nabla$. Let
$$
E_{K_G}\, \subset\, E_G
$$
be an Hermitian structure that gives $\nabla$ and satisfies the condition
$\rho(E_{K_G})\,=\,E_{K_G}$ (it exists by Proposition \ref{prop-i}).

Fix a base point $z_0\,\in\, (E_{K_G})_{x_0}$.
Take any $\gamma\, \in\, \pi_1(X, x_0)$. Let $z_\gamma\, \in\, (E_{K_G})_{x_0}$ be
the point obtained by the parallel translation of $z_0$ along $\gamma$
for the connection $\nabla$. Let
$$
g_\gamma\, \in\, K_G
$$
be the unique element such that $z_0g^{-1}_\gamma\,=\, z_\gamma$.

Now take any $\gamma\, \in\, \eta^{-1}(1)\, \subset\, \Gamma(x_0)$.
Let $y_\gamma\, \in\, (E_{K_G})_{\sigma_X(x_0)}$ be 
the point obtained by the parallel translation of $z_0$ along $\gamma$
for the connection $\nabla$. Let
$$
h'_\gamma\, \in\, K_G
$$
be the unique element such that $z_0\,=\, \rho(y_\gamma)\sigma_G(h'_\gamma)$.
Using the canonical set-theoretic identification of $(\eta')^{-1}(1)$ with $G$
(see \eqref{di} for $\eta'$), the
element $h'_\gamma$ gives an element $h_\gamma\, \in\, (\eta')^{-1}(1)$. Let
\begin{equation}\label{de}
{\delta}\, :\, \Gamma(x_0)\, \longrightarrow\, \widetilde{K}
\end{equation}
be the map that sends any $\gamma\, \in\, \eta^{-1}(0)$ to
$g_\gamma$ constructed above and sends any $\gamma\, \in\, \eta^{-1}(1)$
to $h_\gamma$.

\begin{proposition}\label{prop5}
The function ${\delta}\, :\, \Gamma(x_0)\, \longrightarrow\, \widetilde{K}$
in \eqref{de} lies in ${\rm Hom}_c(\Gamma(x_0)\, , {\widetilde K})$.
\end{proposition}

\begin{proof}
Clearly, $\delta^{-1}(K_G)\,=\, \pi_1(X,x_0)$. In other words,
The diagram as in \eqref{di} is commutative. For any $\gamma\, ,
\gamma'\, \in\, \pi_1(X,x_0)$, it is easy to see that $\delta(\gamma\gamma')\,=\,
\delta(\gamma)\delta(\gamma')$.

Now take $\gamma\, \in\, \pi_1(X,x_0)$ and $\gamma'\, \in\, \eta^{-1}(1)$.
Let $g_{\gamma'}$ (respectively, $g_{\gamma\gamma'}$) be the element of $K_G$
given by $\delta(\gamma')$ (respectively, $\delta(\gamma\gamma')$) using the
set theoretic identification of $(\eta')^{-1}(1)$ with $K_G$ (see
\eqref{di} for $\eta'$). We need to show that
\begin{equation}\label{sm}
g_{\gamma\gamma'} \,=\, \delta(\gamma) g_{\gamma'}\, .
\end{equation}

Let $z'_0 \,\in\, (E_{K_G})_{\sigma_X(x_0)}$ be the parallel translation of $z_0$
along $\gamma'$. Therefore, the parallel translation of $z_0$ along $\gamma
\gamma'$ produces $z'_0\delta(\gamma)^{-1} \,\in\, (E_{K_G})_{\sigma_X(x_0)}$.
Hence,
$$
\rho(z'_0)\,=\,z_0\sigma_G(g^{-1}_{\gamma'}) ~\,~\, \text{ and }
 ~\,~\, \rho(z'_0\delta(\gamma)^{-1})\,=\,z_0\sigma_G(g^{-1}_{\gamma\gamma'})\, .
$$
Since $\rho(yg)\,=\, \rho(y)\sigma_G(g)$, we conclude that
$$
z_0\sigma_G(g^{-1}_{\gamma\gamma'})\,=\,
z_0\sigma_G(g^{-1}_{\gamma'})\sigma_G(\delta(\gamma)^{-1})
\,=\, z_0\sigma_G(g^{-1}_{\gamma'}\delta(\gamma)^{-1})
\, .
$$
Hence $g^{-1}_{\gamma\gamma'}\,=\, g^{-1}_{\gamma'}\delta(\gamma)^{-1}$.
This implies \eqref{sm}.

Hence $g_{\gamma\gamma'} \,=\, \delta(\gamma) g_{\gamma'}$. This coincides with
the corresponding identity in the definition of
${\rm Hom}_c(\Gamma(x_0)\, , {\widetilde K})$.

Now take $\gamma\, \in\, \eta^{-1}(1)$ and $\gamma'\, \in\, \pi_1(X,x_0)$.
Let $g_{\gamma}$ (respectively, $g_{\gamma\gamma'}$) be the element of $K_G$ given by
$\delta(\gamma)$ (respectively, $\delta(\gamma\gamma')$) using the set theoretic
identification of $(\eta')^{-1}(1)$ with $K_G$. We need to show that
\begin{equation}\label{sh2}
g_{\gamma\gamma'}\,=\, g_{\gamma}\sigma_G(\delta(\gamma'))\, .
\end{equation}

Let $z'_0 \,\in\, (E_{K_G})_{\sigma_X(x_0)}$ be the parallel translation of $z_0$
along $\gamma$.

We will compute the parallel translation along the path $\sigma_X(\gamma')
\circ \gamma$ which represents $\gamma\gamma'\,\in\, \Gamma(x_0)$.

Since $\rho$ preserves the connection $\nabla$, the image, under $\rho$, of the parallel
translation along $\gamma'$ is the parallel translation along the loop $\rho(\gamma')$.

Since $z_0$ is taken to $z_0\delta(\gamma')^{-1}$ by the parallel
translation along $\gamma'$, the parallel translation along $\rho(\gamma')$
takes $\rho(z_0)$ to $\rho(z_0)\sigma_G(\delta(\gamma')^{-1})$. We have
\begin{equation}\label{r}
\rho(z'_0)\,=\, z_0\sigma_G(g_\gamma)^{-1}\, .
\end{equation}
Hence $\rho(z_0\sigma_G(g_\gamma)^{-1})\,=\, \rho\circ (z'_0) \,=\,z'_0c$. So,
\begin{equation}\label{c}
\rho(z_0)\,=\, z'_0g_\gamma c\, .
\end{equation}
Since the parallel translation along $\rho(\gamma')$ takes 
$z'_0g_\gamma c$ to $$\rho(z_0)\sigma_G(\delta(\gamma')^{-1})\,=\,
z'_0g_\gamma \sigma_G(\rho(\gamma')^{-1})c\, ,$$ we conclude that
this parallel translation takes $z'_0$ to $z'_0g_\gamma \sigma_G(\delta
(\gamma')^{-1})g^{-1}_\gamma$.

Consequently, the parallel translation along $\rho(\gamma')\circ\gamma$ takes
$z_0$ to $z'_0g_\gamma \sigma_G(\delta(\gamma')^{-1})g^{-1}_\gamma$. Hence
$$
z_0\sigma_G (g_{\gamma\gamma'})^{-1}\,=\,
\rho(z'_0g_\gamma \sigma_G(\delta(\gamma')^{-1})g^{-1}_\gamma)\, .
$$
Therefore, from \eqref{r},
$$
\sigma_G (g_{\gamma\gamma'})^{-1}\,=\, \sigma_G(g_\gamma)^{-1}\sigma_G
(g_\gamma \sigma_G(\delta(\gamma')^{-1})g^{-1}_\gamma)\, .
$$
So we have
$$
\sigma_G (g_{\gamma\gamma'})^{-1}\,=\,
\sigma_G(\sigma_G(\delta(\gamma')^{-1})g^{-1}_\gamma)\, .
$$
This implies \eqref{sh2}.

Finally, take $\gamma\, ,\gamma'\, \in\, \eta^{-1}(1)$. Let $g_{\gamma}$
be as in the previous case. Let $g_{\gamma'}$ be the element of $K_G$
given by $\delta(\gamma')$ using the set theoretic identification of $(\eta')^{-1}(1)$
with $K_G$. We need to show that
\begin{equation}\label{sl}
\delta(\gamma\gamma')\,=\, g_\gamma\sigma_G(g_{\gamma'})c \, .
\end{equation}

Define $z'_0$ as before. From \eqref{r} it follows that 
$$
\rho(z'_0 g_\gamma g^{-1}_{\gamma'})\,=\, z_0\sigma_G(g^{-1}_{\gamma'})\, .
$$
Hence from the definition of $\delta(\gamma')$ we conclude that
$z'_0 g_\gamma g^{-1}_{\gamma'}$ is the parallel translation of
$z_0$ along $\gamma'$.

Therefore, the parallel translation along $\sigma_X(\gamma')$
takes $\rho(z_0)\,=\,z'_0g_\gamma c$ (see \eqref{c}) to
$\rho(z'_0g_\gamma g^{-1}_{\gamma'})\,=\, z_0\sigma_G(g^{-1}_{\gamma'})$
(see \eqref{r}). Hence the parallel translation along $\sigma_X(\gamma')$
takes $z'_0$ to $z_0\sigma_G(g^{-1}_{\gamma'})(g_\gamma)^{-1}c^{-1}$.
Consequently, the parallel translation along the loop
$\sigma_X(\gamma')\circ \gamma$, which represents $\gamma\gamma'
\,\in\, \Gamma(x_0)$, takes $z_0$
to $z_0\sigma_G(g^{-1}_{\gamma'})(g_\gamma)^{-1}c^{-1}$. Hence
$$
\delta(\gamma\gamma')^{-1}\,=\, 
\sigma_G(g^{-1}_{\gamma'})(g_\gamma)^{-1}c^{-1}\, .
$$
This implies \eqref{sl}.
\end{proof}

The above construction of an
element of $\text{Hom}_c(\Gamma(x_0)\, , {\widetilde K})$ from a
polystable pseudo-real principal $G$--bundle of vanishing characteristic classes
of positive degrees is clearly the reverse of the earlier construction of a
flat polystable pseudo-real principal $G$--bundle from an element of
$\text{Hom}_c(\Gamma(x_0)\, , {\widetilde K})$.

Two elements $\delta'\, ,\delta'\, \in\, \text{Hom}_c(\Gamma(x_0)\, , {\widetilde
K})$ are called \textit{equivalent} if there is an element $g\, \in\, K_G$ such
that $\delta'(z)\,=\, g^{-1}\delta(z)g$ for all $z\, \in\, \Gamma(x_0)$.

We have the following:

\begin{theorem}\label{thm3}
There is a natural bijective correspondence between the equivalence classes of
elements of ${\rm Hom}_c(\Gamma(x_0)\, , {\widetilde K})$, and the isomorphism classes
of polystable pseudo-real principal $G$--bundles $(E_G\, ,\rho)$
satisfying the following conditions:
\begin{itemize}
\item $\int_X c_2({\rm ad}(E_G))\wedge\omega^{\dim_{\mathbb C}(X) -2}\,=\, 0$,

\item for any character $\chi$ of $G$, the line bundle over $X$ associated to
$E_G$ for $\chi$ is of degree zero, and

\item the corresponding element in $Z_{\mathbb R}\bigcap K_G$ is $c$
(see Definition \ref{def1}).
\end{itemize}
\end{theorem}

\section{Pseudo-real Higgs $G$--bundles}\label{sec-Higgs}

Consider the differential $d\sigma_X$ in \eqref{e-1}. Using the
natural identification of the holomorphic tangent bundle $TX$ with the real tangent
bundle $T^{\mathbb R}X$, this $d\sigma_X$ produces a $C^\infty$ involution of the
total space of $TX$ over the involution
$\sigma_X$. Since $d\sigma_X\circ J\,=\, - J\circ d\sigma_X$, this
involution of the total space of $TX$ is anti-holomorphic. Let
$$
\widehat{\sigma}\, :\, (TX)^* \,=\, \Omega^1_X\, \longrightarrow\, \Omega^1_X
$$
be the anti-holomorphic involution given by the above involution of $TX$.
Note that $\widehat{\sigma}$ is fiberwise conjugate linear.

Let $(E_G\, ,\rho)$ be a pseudo-real principal $G$--bundle on $X$. The
involution $\widetilde \rho$ of $\text{ad}(E_G)$ in \eqref{rp2} and the
above involution $\widehat{\sigma}$ of $\Omega^1_X$
together produce an anti-holomorphic
involution
\begin{equation}\label{t}
{\widetilde \rho}\otimes \widehat{\sigma}\, :\, \text{ad}(E_G)\otimes
\Omega^1_X\, \longrightarrow\, \text{ad}(E_G)\otimes\Omega^1_X\, .
\end{equation}

A \text{Higgs field} on $(E_G\, ,\rho)$ is a holomorphic section
$$
\theta\, \in\, H^0(X,\, \text{ad}(E_G)\otimes\Omega^1_X)
$$
such that
\begin{itemize}
\item ${\widetilde \rho}\otimes \widehat{\sigma}(\theta)\,=\, \theta$,
where ${\widetilde \rho}\otimes \widehat{\sigma}(\theta)$ is defined in
\eqref{t}, and

\item the holomorphic section $\theta\bigwedge\theta$ of $\text{ad}(E_G)\otimes
\Omega^2_X$ vanishes identically.
\end{itemize}
The above section $\theta\bigwedge\theta$ is defined using the Lie algebra
structure of the fibers of $\text{ad}(E_G)$ and the natural homomorphism
$\Omega^1_X\otimes \Omega^1_X\, \longrightarrow\, \Omega^2_X$.

A \textit{pseudo-real principal Higgs $G$--bundle} is a pseudo-real principal
$G$--bundle equipped with a Higgs field.

Definition \ref{def3} extends as follows:

\begin{definition}\label{def3-2}
A pseudo-real principal Higgs $G$--bundle $(E_G\, ,\rho\, ,\theta)$
over $X$ is called {\em semistable} (respectively, {\em stable})
if for every pair of the form $(U\, ,{\mathfrak p})$, where
\begin{itemize}
\item $\iota_U\, :\, U\, \hookrightarrow\, X$ is a dense open subset
with $\sigma_X(U)\,=\, U$ such that
the complement $X\setminus U$ is a closed complex analytic subset of $X$ of (complex)
codimension at least two,

\item ${\mathfrak p}\, \subsetneq\, {\rm ad}(E_G)\vert_U$ is a parabolic subalgebra
bundle over $U$ such that $\widetilde{\rho}({\mathfrak p})\,=\,{\mathfrak p}$,
and $\iota_{U*} {\mathfrak p}$ is a coherent analytic
sheaf (see Remark \ref{re1}), and

\item $\theta\vert_U\, \in \, H^0(U,\, {\mathfrak p}\otimes\Omega^1_U)$, 
\end{itemize}
we have
${\rm degree}(\iota_{U*}{\mathfrak p})\, \leq\, 0~\, ~\,~{\rm (respectively,~
{\rm degree}(\iota_{U*}{\mathfrak p})\, <\, 0)}$.
\end{definition}

Let ${\mathfrak p}\, \subset\, \text{ad}(E_G)$ be a parabolic subalgebra
bundle such that $\widetilde{\rho}({\mathfrak p})\,=\, {\mathfrak p}$
and $\theta\, \in\, H^0(X,\, {\mathfrak p}\otimes\Omega^1_X)$.
Let $\ell({\mathfrak p})\, \subset\, {\mathfrak p}$ be a Levi subalgebra
bundle such that $\widetilde{\rho}(\ell({\mathfrak p}))\,=\, \ell({\mathfrak p})$
and $\theta\, \in\, H^0(X,\, \ell({\mathfrak p})\otimes\Omega^1_X)$.

The pair $(\ell({\mathfrak p})\, , \theta)$
is called {\em semistable} (respectively, {\em stable})
if for every pair of the form $(U\, ,{\mathfrak q})$, where
\begin{itemize}
\item $\iota_U\, :\, U\, \hookrightarrow\, X$ is a dense open subset
with $\sigma_X(U)\,=\, U$ such that the complement $X\setminus U$
is a closed complex analytic subset of $X$ of (complex) codimension at least two,

\item ${\mathfrak q}\, \subsetneq\, \ell({\mathfrak p})\vert_U$ is a parabolic
subalgebra bundle over $U$ such that $\widetilde{\rho}({\mathfrak q})\,=\,
{\mathfrak q}$, and the direct image $\iota_{U*} {\mathfrak q}$ is a coherent
analytic sheaf, and

\item $\theta\vert_U\, \in\, H^0(U,\, {\mathfrak q}\otimes\Omega^1_U)$,
\end{itemize}
we have
$$
{\rm degree}(\iota_{U*}{\mathfrak q})\, \leq\, 0~\, ~\,~{\rm (respectively,~
{\rm degree}(\iota_{U*}{\mathfrak q})\, <\, 0)}\, .
$$

\begin{definition}\label{def4-2}
A semistable pseudo-real principal Higgs $G$--bundle $(E_G\, ,\rho\, ,\theta)$ 
over $X$ is called {\em polystable} if either $(E_G\, ,\rho)$ is stable,
or there is a proper parabolic subalgebra bundle ${\mathfrak p}\,\subsetneq\,
{\rm ad}(E_G)$, and a Levi subalgebra bundle $\ell({\mathfrak p})\, \subset\,
{\mathfrak p}$, such that the following conditions hold:
\begin{enumerate}
\item $\widetilde{\rho}({\mathfrak p})\, =\, {\mathfrak p}$ and
$\widetilde{\rho}(\ell({\mathfrak p}))\, =\, \ell({\mathfrak p})$,

\item $\theta\, \in\, H^0(X,\, \ell({\mathfrak p})\otimes\Omega^1_X)$, and

\item $(\ell({\mathfrak p}),\, \theta)$ is stable (stability is defined above).
\end{enumerate}
\end{definition}

\begin{lemma}\label{lem-h}
Let $(E_G\, ,\rho\, ,\theta)$ be a semistable pseudo-real principal Higgs $G$--bundle.
Then the principal Higgs $G$--bundle $(E_G\, ,\theta)$ is semistable.

Let $(E_G\, ,\rho\, ,\theta)$ be a polystable pseudo-real principal Higgs $G$--bundle.
Then $(E_G\, ,\theta)$ is polystable.
\end{lemma}

\begin{proof}
We begin by noting that the torsionfree part of the tensor product of
two polystable (respectively, semistable) Higgs sheaves
is again polystable (respectively, semistable); see \cite[p. 553, Lemma 4.4]{BS}
and \cite[p. 553, Proposition 4.5]{BS}. Consequently, Proposition 2.10,
Lemma 2.11 and Corollary 3.8 of \cite{AB} extends to Higgs $G$--bundles; in fact, as
noted at then end of \cite{AB}, once the polystability of the torsionfree part of
the tensor product of polystable Higgs sheaves is established, the results of
\cite{AB} extend to Higgs $G$--bundles. Therefore, the lemma follows exactly
as Lemma \ref{lem1} and Lemma \ref{lem2} do.
\end{proof}

Let $(E_G\, ,\theta)$ be a principal Higgs $G$--bundle on $X$.
Let $E_{K_G}\, \subset\, E_G$ be a $C^\infty$ reduction of structure group
to the maximal compact subgroup $K_G$
(see \eqref{dkg}). The Chern connection on $E_G$ for $E_{K_G}$
will be denoted by $\nabla$, and the curvature of $\nabla$ will be denoted by
${\mathcal K}(\nabla)$. Let $\theta^*$ be the adjoint of $\theta$ with respect to
$E_{K_G}$. To describe $\theta^*$ explicitly, first note that we have a canonical
$C^\infty$ decomposition into a direct sum of real vector bundles
$$
\text{ad}(E_G)\,=\,\text{ad}(E_{K_G})\oplus {\mathcal S}\, ,
$$
where ${\mathcal S}$ is defined in \eqref{S}. If $\theta\,=\, \theta_1+ \theta_2$
with respect to this decomposition, then
\begin{equation}\label{ths}
\theta^*\,=\, -\overline{\theta_1} +\overline{\theta_2}\, .
\end{equation}

\begin{definition}\label{def-h-eh}
The Hermitian structure $E_{K_G}\, \subset\, E_G$ is said to be an
{\em Einstein--Hermitian structure} if
there is an element $\lambda$ in the center of $\mathfrak g$ such that
the section
$$
\Lambda({\mathcal K}(\nabla)+[\theta\, ,\theta^*])\, \in\, C^\infty(X,\, \text{ad}(E_G))
$$
coincides with the one given by $\lambda$; here $\Lambda$ as before
is the adjoint of multiplication by the K\"ahler form. If $E_{K_G}\, \subset\,
E_G$ is an Einstein--Hermitian structure, then the corresponding Chern
connection $\nabla$ is called an {\em Einstein--Hermitian connection}. 
\end{definition}

A principal Higgs $G$--bundle admits an Einstein--Hermitian structure if and only if
it is polystable, and furthermore, the Einstein--Hermitian connection on
a polystable principal Higgs $G$--bundle is unique
\cite{SimpsonJAMS}, \cite{Hitchin}, \cite[p. 554, Theorem 4.6]{BS}.

If $(E_G\, ,\theta)$ is a polystable principal Higgs $G$--bundle such that
\begin{enumerate}
\item $\int_X c_2({\rm ad}(E_G))\wedge\omega^{\dim_{\mathbb C}(X) -2}\,=\, 0$, and

\item for any character $\chi$ of $G$, the line bundle over $X$ associated to
$E_G$ for $\chi$ is of degree zero,
\end{enumerate}
then all the rational characteristic classes of $E_G$ of positive degree vanish
\cite[p. 20, Corollary 1.3]{Si2}.

\begin{proposition}\label{cor-h1}
Let $(E_G\, ,\rho\, ,\theta)$ be a pseudo-real principal Higgs $G$--bundle. Then
the principal Higgs $G$--bundle $(E_G\, ,\theta)$ admits an Einstein--Hermitian
structure $E_{K_G}\, \subset\, E_G$ with $\rho(E_{K_G})\,=\, E_{K_G}$
if and only if $(E_G\, ,\rho\, ,\theta)$ is polystable.
\end{proposition}

\begin{proof}
The proof is similar to the proofs of Corollary \ref{cor1} and Proposition \ref{prop3}.
But the following observation is needed to make the proof of Proposition \ref{prop-i}
work in the present situation (Corollary \ref{cor1} is a consequence of
Proposition \ref{prop-i}).

Let $(F_G\, ,\varphi)$ be a polystable principal Higgs $G$--bundle on $X$. Let
$$
F_{K_G}\, \subset\, F_G
$$
be an Einstein--Hermitian structure on $F_G$. Let $\nabla^F$ be the corresponding
Chern connection on $F_G$. The connection on $\text{ad}(F_G)$ induced by
$\nabla^F$ will be denoted by $\nabla^{\rm ad}$. As in \eqref{S}, let
$$
{\mathcal S}\,:=\, \text{ad}(F_{K_G})^\perp \, \subset\, \text{ad}(F_G)
$$
be the orthogonal complement with respect to an Hermitian structure on $\text{ad}(F_G)$
induced by a $K_G$--invariant Hermitian form on $\mathfrak g$. There is a natural
bijective correspondence between the Hermitian structures on $F_G$ and the smooth
sections of ${\mathcal S}$: the Hermitian structure corresponding to a section $s$
is $\exp(s)(F_{K_G})\, \subset\, F_G$.

An Hermitian structure $\exp(s)(F_{K_G})\, \subset\, F_G$ is an
Einstein--Hermitian structure for $(F_G\, ,\varphi)$ if and only if
\begin{itemize}
\item $s$ is flat with respect to the connection $\nabla^{\rm ad}$ on
$\text{ad}(F_G)$, and

\item $[s\, ,\varphi]\,=\, 0$ (it is the section of $\text{ad}(F_G)\otimes \Omega^1_X$
given by the Lie bracket operation on the fibers of $\text{ad}(F_G)$).
\end{itemize}
Therefore, if $\exp(s)(F_{K_G})$ is an Einstein--Hermitian structure for
$(F_G\, ,\varphi)$, then the Hermitian structure
$$
\exp(s/2)(F_{K_G})\, \subset\, F_G
$$
is also an Einstein--Hermitian structure for $(F_G\, ,\varphi)$.

The rest of the proof of Proposition \ref{prop-i} works as before once the above
observation is incorporated.
\end{proof}

Let $\widetilde{G}\, :=\, G \rtimes ({\mathbb Z}/2{\mathbb Z})$ be
the semi-direct product defined by the involution $\sigma_G$. Consider
$\Gamma(x_0)$ defined in Section \ref{sec4}. Let
$\text{Map}'(\Gamma(x_0)\, , {\widetilde G})$ be the space of all maps
$$
\delta\, :\, \Gamma(x_0)\,\longrightarrow\, {\widetilde G}
$$
such that the following diagram is commutative:
\begin{equation}\label{dil}
\begin{matrix}
e & \longrightarrow & \pi_1(X,x_0)&\longrightarrow &\Gamma(x_0) &
\stackrel{\eta}{\longrightarrow} &
{\mathbb Z}/2{\mathbb Z} &\longrightarrow & e\\
&& \Big\downarrow && ~\Big\downarrow\delta && \Vert\\
e & \longrightarrow & G &\longrightarrow &{\widetilde G} &
\longrightarrow &
{\mathbb Z}/2{\mathbb Z} &\longrightarrow & e
\end{matrix}
\end{equation}

For an element $c\, \in\, Z_{\mathbb R}\bigcap K_G$, let
$\text{Hom}_c(\Gamma(x_0)\, , {\widetilde G})$ be the space of all maps
$$
\delta\, \in\, \text{Map}'(\Gamma(x_0)\, , {\widetilde G})
$$
such that
\begin{itemize}
\item the restriction of $\delta$ to $\pi_1(X,x_0)$ is a homomorphism of groups,

\item $\delta(g'g)\,=\, c\delta(g')\delta(g)$, if $\eta(g)\, =\,1\,=\,
\eta(g')$ (the homomorphism $\eta$ is defined in \eqref{fg}), and

\item $\delta(g'g)\,=\, \delta(g')\delta(g)$ otherwise (meaning
if $\eta(g)\cdot \eta(g')\, =\,0$),
\end{itemize}

If $c\,=\, e$, then $\text{Hom}_c(\Gamma(x_0)\, , {\widetilde G})$
is the space of all homomorphisms from $\Gamma(x_0)$ to ${\widetilde G}$
satisfying \eqref{dil}.

Two elements $\delta'\, ,\delta'\, \in\, \text{Hom}_c(\Gamma(x_0)\, , {\widetilde
G})$ are called \textit{equivalent} if there is an element $g\, \in\, G$ such
that $\delta'(z)\,=\, g^{-1}\delta(z)g$ for all $z\, \in\, \Gamma(x_0)$.

Let $H$ be a connected complex reductive affine algebraic group. A homomorphism
$$
\gamma\, :\, \pi_1(X,x_0)\, \longrightarrow\, H
$$
is called \textit{irreducible} if the image $\gamma(\pi_1(X,x_0))$ is not contained
in some proper parabolic subgroup of $H$. A homomorphism
$$
\gamma\, :\, \pi_1(X,x_0)\, \longrightarrow\, G
$$
is called \textit{completely reducible} if there is a parabolic subgroup $P\, \subset\,
G$ and a Levi factor $L(P)$ of $P$ (see \cite[p. 184]{Hu}, \cite{Bo} for Levi factor)
such that
\begin{itemize}
\item $\gamma(\pi_1(X,x_0))\, \subset\, L(P)$, and

\item the homomorphism $\gamma\, :\, \pi_1(X,x_0)\, \longrightarrow\, L(P)$
is irreducible.
\end{itemize}

A map $$\delta\, \in\, \text{Hom}_c(\Gamma(x_0)\, , {\widetilde G})$$ is called
\textit{completely reducible} if the homomorphism
$\delta\vert_{\pi_1(X,x_0)}$ is completely reducible. Note that if
$\delta$ is is completely reducible, then all elements in
$\text{Hom}_c(\Gamma(x_0)\, , {\widetilde G})$ equivalent to $\delta$ are
also completely reducible.

\begin{proposition}\label{cor-h2}
There is a natural bijective correspondence between the equivalence classes of
completely reducible elements of ${\rm Hom}_c(\Gamma(x_0)\, , {\widetilde G})$, and
the isomorphism classes of polystable pseudo-real principal Higgs $G$--bundles
$(E_G\, ,\rho\, ,\theta)$ satisfying the following conditions:
\begin{itemize}
\item $\int_X c_2({\rm ad}(E_G))\wedge\omega^{\dim_{\mathbb C}(X) -2}\,=\, 0$,

\item for any character $\chi$ of $G$, the line bundle over $X$ associated to
$E_G$ for $\chi$ is of degree zero, and

\item the corresponding element in $Z_{\mathbb R}\bigcap K_G$ is $c$
(see Definition \ref{def1}).
\end{itemize}
\end{proposition}

\begin{proof}
The proof is similar to the proof of Theorem \ref{thm3} after we incorporate
Proposition \ref{cor-h1}. To explain this, take a polystable
pseudo-real principal Higgs $G$--bundle $(E_G\, ,\rho\, ,\theta)$ such that
\begin{itemize}
\item $\int_X c_2({\rm ad}(E_G))\wedge \omega^{\dim_{\mathbb C}(X) -2}\,=\, 0$,

\item for any character $\chi$ of $G$, the line bundle over $X$ associated to
$E_G$ for $\chi$ is of degree zero, and

\item the corresponding element in $Z_{\mathbb R}\bigcap K_G$ is $c$.
\end{itemize}
These conditions imply that all the rational characteristic classes of $E_G$ of
positive degree vanish \cite[p. 20, Corollary 1.3]{Si2}. From Proposition \ref{cor-h1}
we know that $(E_G\, ,\rho\, ,\theta)$ admits an Einstein--Hermitian structure
$E_{K_G}\, \subset\, E_G$ such that $\rho(E_{K_G})\,=\, E_{K_G}$. Let $\nabla^G$ be
the corresponding Chern connection. Define $\theta^*$ as done in \eqref{ths}. Consider
the connection
$$
D\,:=\, \nabla^G+\theta+\theta^*
$$
on $E_G$. It is a flat connection because all the rational characteristic
classes of $E_G$ of positive degree vanish; the monodromy representation for
$D$ is completely reducible \cite[p. 20, Corollary 1.3]{Si2}, \cite[Theorem 1.1]{BG}.

In Theorem \ref{thm3}, consider the construction of an element of
${\rm Hom}_c(\Gamma(x_0)\, , {\widetilde K})$ from a polystable pseudo--real
principal $G$--bundle $F_G$ such that
$\int_X c_2({\rm ad}(F_G))\wedge\omega^{\dim_{\mathbb C}(X) -2}\,=\, 0$
and for any character $\chi$ of $G$, the line bundle over $X$ associated to
$F_G$ for $\chi$ is of degree zero (see Proposition \ref{prop5}). In this construction,
replace the flat Einstein--Hermitian connection $\nabla$ by the flat connection
$D$ constructed above. It yields a completely reducible element
of ${\rm Hom}_c(\Gamma(x_0)\, , {\widetilde G})$.

For the reverse direction, take a completely reducible element
$$
\delta\, \in\, {\rm Hom}_c(\Gamma(x_0)\, , {\widetilde G})\, .
$$
Consider the homomorphism $\delta\vert_{\pi_1(X,x_0)}$. It gives a flat
principal $G$--bundle $(E_G\, ,D)$ and a point $z_0\,\in\,(E_G)_{x_0}$.

In Theorem \ref{thm3}, consider the construction of a pseudo-real principal
$G$--bundle from an element of ${\rm Hom}_c(\Gamma(x_0)\, , {\widetilde K})$
(see Proposition \ref{prop4}). In this construction, replace the
flat Hermitian connection $\nabla$
by the given flat connection $D$ on $E_G$. It yields a pseudo-real structure
\begin{equation}\label{r2}
\rho\, :\, E_G\, \longrightarrow\, E_G
\end{equation}
on the principal $G$--bundle $E_G$.

Since the monodromy representation for $D$ is completely reducible, a theorem
of Corlette says that $E_G$ admits a harmonic reduction
$$
E_{K_G}\, \subset\, E_G
$$
(see \cite[p. 368, Theorem 3.4]{Co}, \cite[p. 19, Theorem 1]{Si2}).
We will show that the harmonic reduction $E_{K_G}$ can be so chosen
that it satisfies the condition
\begin{equation}\label{ih}
\rho(E_{K_G})\,=\, E_{K_G}\, ,
\end{equation}
where $\rho$ is the pseudo-real structure obtained in \eqref{r2}.

To prove this, take a harmonic reduction $E_{K_G}\, \subset\, E_G$. As in \eqref{S},
let
$$
{\mathcal S}\,:=\, \text{ad}(E_{K_G})^\perp \, \subset\, \text{ad}(E_G)
$$
be the orthogonal complement with respect to an Hermitian structure on $\text{ad}(E_G)$
induced by a $K_G$--invariant Hermitian form on $\mathfrak g$.
We recall that every Hermitian structures on $E_G$ is of the form
$\exp(s)(E_{K_G})$, where $s$ is a smooth sections of ${\mathcal S}$.

Let $D^{\rm ad}$ be the flat connection on $\text{ad}(E_G)$ induced by the
connection $D$ on $E_G$. An Hermitian structure
$$
\exp(s)(E_{K_G})\, \subset\, E_G
$$
is a harmonic reduction for $(E_G\, ,D)$ if and only if
$$
D^{\rm ad}(s)\,=\, 0\, .
$$
Therefore, if $\exp(s)(E_{K_G})\, \subset\, E_G$ is a harmonic reduction for
$(E_G\, ,D)$, then
$$
\exp(s/2)(E_{K_G})\, \subset\, E_G
$$
is also a harmonic reduction for $(E_G\, ,D)$.
Now the proof of Proposition \ref{prop-i} gives that there is a
harmonic reduction $E_{K_G}$ for $(E_G\, ,D)$ such that \eqref{ih} holds.

Let $(E'_G\, ,\theta)$ be the principal Higgs $G$--bundle corresponding
to the triple $(E\, ,D\, ,E_{K_G})$,
where $E_{K_G}$ satisfies \eqref{ih}. So
$$
D\,=\, \nabla+\theta+\theta^*\,=\,\nabla^{1,0}+\nabla^{0,1}+\theta+\theta^*\, ,
$$
such that the following three conditions hold:
\begin{enumerate}
\item $\nabla$ is a connection on $E_G$ coming from a connection on $E_{K_G}$.

\item $\nabla^{0,1}\circ \nabla^{0,1}\,=\, 0$, meaning $\nabla^{0,1}$ defines a
holomorphic structure on the $C^\infty$ principal $G$--bundle $E_G$. This
holomorphic principal $G$--bundle $(E_G\, ,\nabla^{0,1})$ is denoted by $E'_G$.

\item $\theta$ is a Higgs field on the holomorphic principal $G$--bundle
$E'_G$.
\end{enumerate}
(See \cite[p. 13]{Si2}.) The triple $(E'_G\, ,\rho\, ,\theta)$, where
$\rho$ is constructed in \eqref{r2}, is a polystable pseudo-real principal
Higgs $G$--bundle.
\end{proof}


\end{document}